\documentclass{amsart}
\usepackage{amscd}
\newtheorem{theorem}{Theorem}[section]
\newtheorem{lemma}[theorem]{Lemma}
\newtheorem{proposition}[theorem]{Proposition}
\newtheorem{corollary}[theorem]{Corollary}
\theoremstyle{definition}
\newtheorem{definition}[theorem]{Definition}
\newtheorem{example}[theorem]{Example}

\theoremstyle{remark}
\newtheorem{remark}[theorem]{Remark}

\numberwithin{equation}{section}

\theoremstyle{plain}

\newtheorem{problem}[theorem]{\bf Problem}

\def\int{\mathop{\roman{int}}}

\def\1{^{-1}}

\def\proof{{\bf Proof. }}
\def\endproof{\hfill \qed}

\numberwithin{equation}{section}
\begin{document}

\title[
Covering maps for locally path-connected spaces
]
   {Covering maps for locally path-connected spaces}

\author{N.~Brodskiy}
\address{University of Tennessee, Knoxville, TN 37996}
\email{brodskiy@@math.utk.edu}

\author{J.~Dydak}
\address{University of Tennessee, Knoxville, TN 37996}
\email{dydak@@math.utk.edu}

\author{B.~Labuz}
\address{University of Tennessee, Knoxville, TN 37996}
\email{labuz@@math.utk.edu}

\author{A.~Mitra}
\address{University of Tennessee, Knoxville, TN 37996}
\email{ajmitra@@math.utk.edu}

\keywords{covering maps, locally path-connected spaces}

\subjclass[2000]{Primary 55Q52; Secondary 55M10, 54E15}
\date{February 14, 2008.}

\begin{abstract}
We define Peano covering maps and prove basic properties analogous to classical covers. Their domain is always locally path-connected but the range may be an arbitrary
topological space. One of characterizations of Peano covering maps is via the uniqueness
of homotopy lifting property for all locally path-connected spaces.

Regular Peano covering maps over path-connected spaces are shown to be identical
with generalized regular covering maps introduced by 
Fischer and Zastrow \cite{FisZas}.
If $X$ is path-connected, then every Peano covering map is equivalent
to the projection $\widetilde X/H\to X$,
where $H$ is a subgroup of the fundamental group of $X$ and $\widetilde X$ equipped with the topology used in \cite{BogSie}, \cite{FisZas}
and introduced in \cite[p.82]{Spa}.
The projection $\widetilde X/H\to X$ is a Peano covering map if and only if
it has the unique path lifting property.
We define a new topology on $\widetilde X$ for which
one has a characterization of $\widetilde X/H\to X$ having the unique path lifting
property if $H$ is a normal subgroup of $\pi_1(X)$. Namely, $H$ must be closed
in $\pi_1(X)$. Such groups include $\pi(\mathcal{U},x_0)$
($\mathcal{U}$ being an open cover of $X$) and the kernel of the natural
homomorphism $\pi_1(X,x_0)\to \check\pi_1(X,x_0)$.
\end{abstract}
\maketitle

\medskip
\medskip
\tableofcontents

\section{Introduction}

As locally complicated spaces naturally appear in mathematics (examples:
boundaries of groups, limits under Gromov-Hausdorff convergence)
there is an effort to extend homotopy-theoretical concepts to such spaces.
This paper is devoted to a theory of coverings by locally path-connected
spaces. Zeeman's example \cite[6.6.14 on p.258]{HilWyl} demonstrates
difficulty in constructing a theory of coverings by non-locally path-connected
spaces (that example amounts to two non-equivalent classical coverings
with the same image of the fundamental groups).
For coverings in the uniform category see \cite{BP3} and \cite{BDLM1}.

To simplify exposition let us introduce the following concepts:

\begin{definition}\label{PeanoSpacesDef}
A topological space $X$ is an {\bf lpc-space}
if it is locally path-connected.
$X$ is a {\bf Peano space}
if it is locally path-connected and connected.
\end{definition}

Fischer and Zastrow \cite{FisZas} defined {\bf generalized regular coverings}
of $X$
as functions $p\colon \bar X\to X$ satisfying the following conditions
for some normal subgroup $H$ of $\pi_1(X)$:
\begin{itemize}
\item[R1.]  $\bar X$ is a Peano space.
\item[R2.] The map $p\colon \bar X\to X$ is a continuous surjection
and $\pi_1(p)\colon \pi_1(\bar X)\to \pi_1(X)$ is a monomorphism onto $H$.
\item[R3.] For every Peano space $Y$, for
every continuous function $f\colon (Y, y)\to (X, x_0)$ with $f_\ast(\pi_1(Y, y)) \subset H$, and for
every $\bar x \in \bar X$ with $p(\bar x) = x_0$, there is a unique continuous
$g\colon (Y,y)\to (\bar X,\bar x))$ with $p \circ  g = f$.
\end{itemize}

Our view of the above concept is that of being universal in a certain class of maps
and we propose a different way of defining covering maps
between Peano spaces in Section \ref{SECTION Peano-coverings}.
\par Our first observation is that each path-connected space
$X$ has its universal Peano space $P(X)$, the set $X$ equipped with new topology,
such that the identity function $P(X)\to X$ corresponds to a generalized regular
covering for $H=\pi_1(X)$. That way quite a few results in the literature
can be formally deduced from earlier results for Peano spaces.
\par The way the projection $P(X)\to X$ is characterized in \ref{LPCExistsThm}  generalizes to the concept
of {\bf Peano maps} in Section \ref{SECTION Peano-coverings}
and our {\bf Peano covering maps} combine Peano maps with two classical concepts:
Serre fibrations and unique path lifting property.
Peano covering maps possess several properties analogous to the classical
covering maps \cite{Lim} (example: local Peano covering maps are Peano covering maps).
One of them is that they are all quotients $\widehat X_H$ of
the universal path space $\widetilde X$ equipped with the topology
defined in the proof of Theorem 13 on p.82 in \cite{Spa} and used successfully
by Bogley-Sieradski \cite{BogSie} and Fischer-Zastrow \cite{FisZas}.
It turns out the endpoint projection $\widehat X_H\to X$ is a Peano covering
map if and only if it has the uniqueness of path lifts property (see \ref{ProjIsPeanoCMCharThm}).
In an effort to unify Peano covering maps with uniform covering maps of
\cite{BP3} and \cite{BDLM1} (we will explain the connection in \cite{BDLM2})
we were led to a new topology on $\widetilde X_H$ (see Section \ref{SECTION: BS topology}). Its main advantage is that there is a necessary and sufficient condition
for $\widetilde X_H\to X$ to have the unique path lifting property in case $H$ is a normal
subgroup of $\pi_1(X)$. It is $H$ being closed in $\pi_1(X)$.
That explains Theorem 6.9 of \cite{FisZas} as the basic groups there turn out
to be closed in $\pi_1(X)$. As an application of our approach we
show existence of a universal Peano covering map over a given path-connected space.

\par
We thank Sasha Dranishnikov for bringing the work of Fischer-Zastrow \cite{FisZas}
to our attention. We thank Greg Conner, Katsuya Eda, Ale\v s Vavpeti\' c, and Ziga Virk for helpful comments.

\section{Constructing Peano spaces}\label{SECTION: UPeanoSpace}

The purpose of this section is to discuss various ways of constructing
new Peano spaces.

\subsection{Universal Peano space}
In analogy to the universal covering spaces we introduce the following notion:

\begin{definition}\label{LPCspacesDef}
Given a topological space $X$ its {\bf universal lpc-space}
$lpc(X)$ is an lpc-space together with a continuous map (called the {\bf universal Peano map})
$\pi\colon lpc(X)\to X$ satisfying the following universality condition:
\par For any map $f\colon Y\to X$ from an lpc-space $Y$
there is a unique continuous lift $g\colon Y\to lpc(X)$ of $f$ (that means $\pi\circ g=f$).
\end{definition}

\begin{theorem}\label{LPCExistsThm}
Every space $X$ has a universal lpc-space. It is homeomorphic to the set $X$ equipped
with a new topology, the one generated by all path-components of all open subsets of the existing topology of $X$. 
\end{theorem}
\proof Let $U$ be an open set in $X$ containing the point $x$ and $c(x,U)$ be the path component of $x$ in $U$.
Since $z\in c(x,U)\cap c(y,V)$ implies $c(z,U\cap V)\subset c(x,U)\cap c(y,V)$,
the family $\{c(x,U)\}$, where $U$ ranges over all open subsets of $X$ and $x$
ranges over all elements of $U$, forms a basis.

Given a map $f\colon Y\to X$ and given an open set $U$ of $X$ containing $f(y)$
one has $f(c(y,f^{-1}(U)))\subset c(f(y),U)$. That proves $f\colon Y\to lpc(X)$
is continuous if $Y$ is an lpc-space. It also proves $lpc(X)$ is locally path-connected
as any path in $X$ induces a path in $lpc(X)$.
\endproof

\begin{remark}\label{ReftoFisZas4.17}
The topology above was mentioned in Remark 4.17 of \cite{FisZas}.  
After the first version of this paper was written we were informed by Greg Conner
of his unpublished preprint \cite{ConFea} with David Fearnley, where that topology
is discussed and its properties (compactness, metrizability) are investigated.
\end{remark}

If $X$ is path-connected, then $lpc(X)$ is a {\bf universal Peano space} $P(X)$
in the following sense: given a map $f\colon Z\to X$ from a Peano space $Z$
to $X$ there is a unique lift $g\colon Z\to P(X)$ of $f$.

In the remainder of this section we give sufficient conditions for a function
on an lpc-space to be continuous. Those conditions are in terms of maps
from basic Peano spaces: the arc in the first-countable case and {\bf hedgehogs}
(see Definition  \ref{GeneralizedHEarrings})
in the arbitrary case.

\begin{proposition}\label{MapsFromPeanoToFirstCountable}
Suppose $f\colon Y\to X$ is a function from a first-countable lpc-space $Y$. $f$ is continuous if $f\circ g$ is continuous for every path $g\colon I\to Y$
in $Y$.
\end{proposition}
\proof Suppose $U$ is open in $X$. It suffices to show that for each $y\in f^{-1}(U)$
there is an open set $V$ in $Y$ containing $y$ such that the path component
of $y$ in $V$ is contained in $f^{-1}(U)$. Pick a basis of neighborhoods $\{V_n\}_{n\ge 1}$
of $y$ in $Y$ and assume for each $n\ge 1$ there is a path $\alpha_n$ in $V_n$
joining $y$ to a point $y_n\notin f^{-1}(U)$. Those paths can be spliced to one path
$\alpha$
from $y$ to $y_1$ and going through all points $y_n$, $n\ge 2$.
$f\circ \alpha$ starts from $f(y)$ and goes through all points $f(y_n)$, $n\ge 1$.
However, as $U$ is open, it must contain almost all of them, a contradiction.
\endproof

The construction of the topology on $lpc(X)$ in \ref{LPCExistsThm} can be
done in the spirit of the finest topology on $X$ that retains
the same continuous maps from a class of spaces. 

\begin{proposition}\label{UniversalArcLiftingConstruction}
Suppose $X$ is a path-connected topological space
and $\mathcal{P}$ is a class of Peano spaces.
The family $\mathcal{T}$ of subsets $U$ of $X$ such that
$f^{-1}(U)$ is open in $Z\in\mathcal{P}$ for any map $f\colon Z\to X$ in the original topology,
is a topology and $\mathcal{P}(X):=(X,\mathcal{T})$ is a Peano space.
\end{proposition}
\proof Since $f^{-1}(U\cap V)=f^{-1}(U)\cap f^{-1}(V)$, $\mathcal{T}$
is a topology on $X$. Suppose $U\in\mathcal{T}$ and $C$ is a path component
of $U$ in the new-topology.
Suppose $f\colon Z\to X$ is a map and $f(z_0)\in C$. As $f^{-1}(U)$ is open,
there is a connected neighborhood $V$ of $z_0$ in $Z$ satisfying $f(V)\subset U$.
As $f(V)$ is path-connected, $f(V)\subset C$ and $C\in \mathcal{T}$.
\endproof

In case of first-countable spaces $X$ we have a very simple characterization of the universal
Peano map of $X$:

\begin{corollary}\label{PeanoForFirstCountable}
If $X$ is a first-countable path-connected topological space, then a map $f\colon Y\to X$
is a universal Peano map if and only if $Y$ is a Peano space, $f$ is a bijection, and $f$
has the path lifting property.
\end{corollary}
\proof Consider $\mathcal{A}(X)$ as in \ref{UniversalArcLiftingConstruction},
where $\mathcal{A}$ consists of the unit interval.
Notice the identity function $P(X)\to \mathcal{A}(X)$ is continuous as $P(X)$ is first-countable
(use \ref{MapsFromPeanoToFirstCountable}). Since the topology on $\mathcal{A}(X)$ is finer
than that on $P(X)$, $P(X)=\mathcal{A}(X)$.
Since $f$ induces a homeomorphism from $\mathcal{A}(Y)$ to $\mathcal{A}(X)$ (due to the uniqueness of
path lifting property of $f$),
the composition $\mathcal{A}(Y)\to \mathcal{A}(X)\to P(X)$ is a homeomorphism and $f\colon Y\to P(X)$
must be a homeomorphism (its inverse is $P(X)\to \mathcal{A}(Y)\to Y$).
\endproof

The construction in \ref{UniversalArcLiftingConstruction} can be used
to create counter-examples to \ref{PeanoForFirstCountable} in case $X$ is not first-countable.

\begin{example}\label{ExampleOfNonFirstCountable}
Let $X$ be the cone over an uncountable discrete set $B$.
Subsets of $X$ that miss the vertex $v$ are declared open if and only if
they are open in the CW topology on $X$.  A subset $U$ of $X$ that contains $v$ is declared open
if and only if $U$ contains all but countably many edges of the cone
and $U\setminus\{v\}$ is open in the CW topology on $X$ (that means $X$ is a hedgehog
if $B$ is of cardinality $\omega_1$ - see \ref{GeneralizedHEarrings}).
Notice $\mathcal{A}(X)$ is $X$ equipped with the CW topology,
the identity function $\mathcal{A}(X)\to X$ has the path lifting property but is not a homeomorphism.
\end{example}
\proof Notice every subset of $X\setminus\{v\}$ that meets each edge in at most one point
is discrete. Hence a path in $X$ has to be contained in the union of finitely many
edges. That means $\mathcal{A}(X)$ is $X$ with the CW topology.
\endproof

We generalize \ref{ExampleOfNonFirstCountable} as follows:

\begin{definition}\label{GeneralizedHEarrings}
A {\bf generalized Hawaiian Earring} is the wedge 
\par\noindent 
$(Z,z_0)=\bigvee\limits_{s\in S} (Z_s,z_s)$ of pointed Peano spaces indexed by a directed set $S$ and equipped
with the following topology (all wedges in this paper are considered with that particular
topology):
\begin{enumerate}
\item $U\subset Z\setminus\{z_0\}$ is open if and only if
$U\cap Z_s$ is open for each $s\in S$,
\item $U$ is an open neighborhood of $z_0$ if and only if 
there is $t\in S$ such that $Z_s\subset U$ for all $s > t$
and $U\cap Z_s$ is open for each $s\in S$.
\end{enumerate}
A {\bf hedgehog} is a generalized Hawaiian Earring
$(Z,z_0)=\bigvee\limits_{s\in S} (Z_s,z_s)$ such that each $(Z_s,z_s)$
is homeomorphic to $(I,0)$.
\end{definition}

Our definition of generalized Hawaiian Earrings is different from the definition of Cannon and Conner~\cite{CanCon Big}. Also, the preferred terminology in \cite{CanCon Big}
is that of a {\bf big Hawaiian Earring}.

Observe each generalized Hawaiian Earring is a Peano space.

\begin{lemma}\label{BasicHedgeHogLemma}
Let $S$ be a basis of neighborhoods of $x_0$ in $X$ ordered by inclusion
(i.e., $U \leq V$ means $V\subset U$).
If, for each $U\in S$, $\alpha_U\colon I\to U$ is a path in $U$ starting
from $x_0$, then their wedge 
$$\bigvee\limits_{U\in S}\alpha_U\colon \bigvee\limits_{U\in S}(I_U,0_U)\to (X,x_0)$$
is continuous, where $(I_U,0_U)=(I,0)$ for each $U\in S$.
\end{lemma}
\proof Only the continuity of $g=\bigvee\limits_{U\in S}\alpha_U$
at the base-point of the hedgehog $\bigvee\limits_{U\in S}(I_U,0_U)$
is not totally obvious. However, if $V$ is a neighborhood of $x_0$ in $X$,
then $g^{-1}(V)$ contains all $I_U$ if $U\subset V$ and $g^{-1}(V)\cap I_W$
is open in $I_W$ for all $W\in S$.
\endproof

\begin{proposition}\label{MapsFromPeanoToArbitrary}
Suppose $f\colon Y\to X$ is a function from an lpc-space $Y$. 
$f$ is continuous if $f\circ g$ is continuous for every map $g\colon Z\to Y$
from a hedgehog $Z$ to $Y$.
\end{proposition}
\proof Assume $U$ is open in $X$ and $x_0=f(y_0)\in U$.
Suppose for each path-connected neighborhood $V$ of $y_0$ in $Y$ there is a path
$\alpha_V\colon (I,0)\to (V,y_0)$ such that $\alpha_V(1)\notin f^{-1}(U)$.
By \ref{BasicHedgeHogLemma} the wedge $g=\bigvee\limits_{V\in S}\alpha_V$
is a map $g$ from a hedgehog to $Y$ (here $S$ is the family of all path-connected neighborhoods
of $y_0$ in $Y$). Hence $h=f\circ g$ is continuous
and there is $V\in S$ so that $I_V\subset h^{-1}(U)$.
That means $f(\alpha_V(I))\subset U$, a contradiction.
\endproof

\subsection{Basic topology on $\widetilde X$}

The philosophical meaning of this section is that many results can be reduced
to those dealing with Peano spaces via the universal Peano space construction. Let us illustrate this point of view by discussing
a topology on $\widetilde X$.

Suppose $(X,x_0)$ is a pointed topological space.
Consider the space $\widetilde X$ of homotopy classes of paths
in $X$ originating at $x_0$.
It has an interesting topology (see the proof of Theorem 13 on p.82 in \cite{Spa})
that has been put to use in \cite{BogSie} and \cite{FisZas}. Its basis consists of
sets $B([\alpha],U)$ ($U$ is open in $X$, $\alpha$ joins $x_0$
and $\alpha(1)\in U$) defined as follows: $[\beta]\in B([\alpha],U)$ if and only if there is a path $\gamma$
in $U$ from $\alpha(1)$ to $\beta(1)$ such that $\beta$ is homotopic rel. endpoints to
the concatenation $\alpha\ast\gamma$.

$\widetilde X$ equipped with the above topology will be denoted by
$\widehat X$ as in \cite{BogSie}.

Both \cite{BogSie} and \cite{FisZas} consider quotient spaces $\widehat X/H$,
where $H$ is a subgroup of $\pi_1(X,x_0)$. We find it more convenient to follow
\cite[pp.82-3]{Spa}:

\begin{definition}\label{BSModHDef}
Suppose $H$ is a subgroup of $\pi_1(X,x_0)$.
Define $\widetilde X_H$ as the set of equivalence classes of paths
in $X$ under the relation $\alpha\sim_H\beta$ defined via $\alpha(0)=\beta(0)=x_0$,
$\alpha(1)=\beta(1)$ and $[\alpha\ast\beta^{-1}]\in H$
(the equivalence class of $\alpha$ under the relation $\sim_H$ will be denoted by $[\alpha]_H$).
\end{definition}

To introduce a topology on $\widetilde X_H$ we define sets $B_H([\alpha]_H,U)$
(denoted by $<\alpha,U>$ on p.82 in \cite{Spa}), where $U$ is open in $X$, $\alpha$ joins $x_0$
and $\alpha(1)\in U$, as follows: $[\beta]_H\in B_H([\alpha]_H,U)$ if and only if there is a path $\gamma$
in $U$ from $\alpha(1)$ to $\beta(1)$ such that $[\beta\ast (\alpha\ast\gamma)^{-1}]\in H$ (equivalently, $\beta\sim_H \alpha\ast\gamma$).

$\widetilde X_H$ equipped with the topology (which we call the
{\bf basic topology on $\widetilde X_H$})
whose basis consists of $B_H([\alpha]_H,U)$, where $U$ is open in $X$, $\alpha$ joins $x_0$
and $\alpha(1)\in U$, is denoted by $\widehat X_H$
in analogy to the notation $\widehat X$ in \cite{BogSie} that corresponds
to $H$ being trivial.

Given a path $\alpha$ in $X$ and a path $\beta$ in $X$ from $x_0$ to $\alpha(0)$ one can define
a {\bf standard lift} $\hat \alpha$ of it to $\widehat X_H$ originating at $[\beta]_H$ by the formula
$\hat \alpha(t)=[\beta \ast\alpha_t]_H$, where $\alpha_t(s)=\alpha(s\cdot t)$
for $s,t\in I$ (see \cite[Proposition 6.6.3]{HilWyl}).

Let us extract the essence of the proof of \cite[Theorem 13 on pp.82--83]{Spa}:

\begin{lemma}\label{SemiSCLemma}
Suppose $X$ is a path-connected space and $H$ is a subgroup of $\pi_1(X,x_0)$.
An open set $U\subset X$ is evenly covered by $p_H\colon \widehat X_H\to X$
if and only if $U$ is locally path-connected and the image of $h_\alpha\colon \pi_1(U,x_1)\to \pi_1(X,x_0)$
is contained in $H$ for any path $\alpha$ in $X$ from $x_0$ to any $x_1\in U$.
\end{lemma}
\proof Recall that $U$ is {\bf evenly covered} by $p_H$ (see \cite[p.62]{Spa}) if $p_H^{-1}(U)$
is the disjoint union of open subsets $\{U_s\}_{s\in S}$ of $\widehat X_H$
each of which is mapped homeomorphically onto $U$ by $p_H$.
Also, recall $h_\alpha\colon \pi_1(U,x_1)\to \pi_1(X,x_0)$ is given by
$h_\alpha([\gamma])=[\alpha\ast\gamma\ast\alpha^{-1}]$.

Suppose $U$ is evenly covered, $\gamma$ is a loop in $(U,x_1)$,
and $\alpha$ is a path from $x_0$ to $x_1$.
If $[\alpha]_H\ne [\alpha\ast\gamma]_H$, then they belong
to two different sets $U_u$ and $U_v$, $u,v\in S$. However, there is a path from
$[\alpha]_H$ to $[\alpha\ast\gamma]_H$ in $p_H^{-1}(U)$ given by the standard lift of $\gamma$, a contradiction.
Thus $[\alpha]_H= [\alpha\ast\gamma]_H$ and
$[\alpha\ast\gamma\ast\alpha^{-1}]\in H$.

To show that $U$ is locally path-connected, take a point $x_1\in U$, pick
a path $\alpha$ from $x_0$ to $x_1$ and select the unique $s\in S$
so that $[\alpha]_H\in U_s$. There is an open subset $V$ of $U$ satisfying
$B_H([\alpha]_H,V)\subset U_s$. As $p_H|U_s$ maps $U_s$ homeomorphically
onto $U$, $p_H(B_H([\alpha]_H,V))$ is an open neighborhood of $x_1$ in $U$
and it is path-connected.

Suppose $U$ is locally path-connected and the image of $h_\alpha\colon \pi_1(U,x_1)\to \pi_1(X,x_0)$
is contained in $H$ for any path $\alpha$ in $X$ from $x_0$ to any $x_1\in U$.
Pick a path component $V$ of $U$ and notice sets $B_H([\beta]_H,V)$,
$\beta$ ranging over paths from $x_0$ to points of $V$, are either identical or disjoint.
Observe $p_H|B_H([\beta]_H,V)$ maps $B_H([\beta]_H,V)$ homeomorphically
onto $V$. Thus each $V$ is evenly covered and that is sufficient to conclude $U$
is evenly covered.
\endproof

\par As in \cite[p.81]{Spa}, given an open cover $\mathcal{U}$ of $X$,
$\pi(\mathcal{U},x_0)$ is the subgroup of $\pi_1(X,x_0)$
generated by elements of the form $[\alpha\ast\gamma\ast\alpha^{-1}]$,
where $\gamma$ is a loop in some $U\in\mathcal{U}$
and $\alpha$ is a path from $x_0$ to $\gamma(0)$.

Here is our improvement of \cite[Theorem 13 on p.82]{Spa} and \cite[Theorem 6.1]{FisZas}:

\begin{theorem}\label{BasicThmOnWidehatX}
If $X$ is a path-connected space and $H$ is a subgroup of $\pi_1(X,x_0)$,
then the endpoint projection $p_H\colon \widehat X_H\to X$
is a classical covering map if and only if $X$ is a Peano space and there is an open covering
$\mathcal{U}$ of $X$ so that $\pi(\mathcal{U},x_0)\subset H$.
\end{theorem}
\proof Apply \ref{SemiSCLemma}.
\endproof

\begin{proposition}\label{UniversalCoveringForXAndPX}
$\widehat {P(X)}_H$
is naturally homeomorphic to $\widehat X_H$ if $X$ is path-connected.
\end{proposition}
\proof Since continuity of $f\colon (Z,z_0)\to (P(X),x_0)$, for any Peano space $Z$,
is equivalent to the continuity of $f\colon (Z,z_0)\to (X,x_0)$, paths in $(P(X),x_0)$
correspond to paths in $(X,x_0)$. Also, $\pi_1(P(X),x_0)\to \pi_1(X,x_0)$
is an isomorphism so $H$ is a subgroup of both $\pi_1(P(X),x_0)$ and  $\pi_1(X,x_0)$,
and the equivalence classes of relations $\sim_H$ are identical in both spaces
$\widetilde {P(X)}$ and $\widetilde X$. Notice that basis open sets are identical in $\widehat {P(X)}_H$ and $\widehat X_H$.
\endproof

\begin{remark}\label{ApplyingPXRem}
In view of \ref{UniversalCoveringForXAndPX} some results in \cite{FisZas} dealing
with maps $f\colon Y\to X$, where $Y$ is Peano, can be derived formally from corresponding
results for $f\colon Y\to P(X)$. A good example is Lemma 2.8
in \cite{FisZas}:
\par $p\colon \tilde X \to X$ has the unique path lifting property if and only if
$\tilde X$ is
simply connected.
\par It follows formally from Corollary 4.7 in \cite{BogSie}:
\par The universal endpoint projection $p\colon \hat Z \to Z$ for a connected and locally
path-connected space $Z$ has the unique path lifting property if and only if $\hat Z$ is simply connected.
\end{remark}

When working in the pointed topological category the space $\widehat X_H$
is equipped with the base-point $\widehat x_0$ equal to the equivalence class
of the constant path at $x_0$.

Let us illustrate $\widehat X_H$
in the case of $H=\pi_1(X,x_0)$.

\begin{proposition}\label{CaseHBeingWholeGroup}
If $H=\pi_1(X,x_0)$, then
\begin{itemize}
\item[a.] The endpoint projection
$p_H\colon (\widehat X_H,\widehat x_0)\to (X,x_0)$
is an injection and
\par\noindent
$p_H(B([\alpha]_H,U))$ is the path component of $\alpha(1)$ in $U$,
\item[b.] $\widehat X_H$
is a Peano space,
\item[c.]  Given a map $g\colon (Z,z_0)\to (X,x_0)$ from a pointed Peano space to
$(X,x_0)$, there is a unique lift $h\colon (Z,z_0)\to (\widehat X_H,\widehat x_0)$ of $g$
($p_H\circ h=g$).
\end{itemize}
\end{proposition}
\proof a). Clearly, $p_H(B_H([\alpha]_H,U))$ equals path component of $\alpha(1)$ in $U$.
If $[\beta_1]_H$ and $[\beta_2]_H$ map to the same point $x_1$,
then $\beta_1(1)=\beta_2(1)$ and $\gamma=\beta_1\ast\beta_2^{-1}$ is a loop.
Hence $[\gamma]\in H$ and $[\beta_2]_H=[\gamma\ast\beta_2]_H=[\beta_1]_H$
proving $p_H$ is an injection.
\par b) is well-established in both \cite{BogSie} and \cite{FisZas}. Notice it follows from a).
\par c). For each $z\in Z$ pick a path $\alpha_z$ from $z_0$ to $z$ in $Z$.
Define $h(z)$ as $[\alpha_z]_H$ and notice $h$ is continuous as 
$h^{-1}(B_H([\alpha_z]_H,U))$ equals the path component of $g^{-1}(U)$ containing $z$
(use Part a)). As $p_H$ is injective, there is at most one lift of $g$.
\endproof

In view of \ref{CaseHBeingWholeGroup} we have a convenient definition
of a universal Peano space in the pointed category:

\begin{definition}\label{PeanoSpacePointedDef}
By the {\bf universal Peano space} $P(X,x_0)$ of $(X,x_0)$
we mean the pointed space $(\widehat X_H,\widehat x_0)$, $H=\pi_1(X,x_0)$,
and the {\bf universal Peano map of $(X,x_0)$}
is the endpoint projection $P(X,x_0)\to (X,x_0)$.
Equivalently, $P(X,x_0)$ is $(P(C),x_0)$, where $C$ is the path component of $x_0$ in $X$.
\end{definition}

Due to standard lifts the endpoint projection
 $p_H\colon \widehat X_H\to (X,x_0)$ always has the path lifting property.
 Thus the issue of interest is the uniqueness of path lifting property of $p_H$.

Here is a necessary and sufficient condition for $p_H$ to have the
unique path lifting property (compare it to \cite[Theorem 4.5]{BogSie} for
Peano spaces):

\begin{proposition}\label{ProjHasUPLP}
If $X$ is a path-connected space and $x_0\in X$, then the following conditions are equivalent:
\begin{itemize}
\item[a.] $p_H\colon (\widehat X_H,\widehat x_0)\to (X,x_0)$ has the unique path lifting property,
\item[b.]  The image of $\pi_1(p_H)\colon \pi_1(\widehat X_H,\widehat x_0)\to \pi_1(X,x_0)$ is contained in $H$.
\end{itemize}
\end{proposition}
\proof a)$\implies$b). Given a loop $\alpha$ in $\widehat X_H$
it must equal the standard lift of $\beta=p_H(\alpha)$. For the standard lift
of $\beta$ to be a loop in $\widehat X_H$ one must have $[\beta]\in H$.
\par
b)$\implies$a) Given a lift $\bar \alpha$ of a path $\alpha$ in $(X,x_0)$
it suffices to show $\bar \alpha(1)=[\alpha]_H$ as that implies $\bar \alpha$
is the standard lift of $\alpha$ (use $\alpha|[0,t]$ instead of $\alpha$).
Pick a path $\beta$ satisfying $\hat\alpha(1)=[\beta]_H$
and let $\hat\beta$ be its standard lift. As $\bar\alpha\ast(\hat\beta)^{-1}$
is a loop in $\widehat X_H$, its image $\gamma=p_H(\bar\alpha\ast(\hat\beta)^{-1})$
generates an element $[\gamma]$ of $H$.
Hence $\alpha\sim \gamma\ast\beta$ and $\bar\alpha(1)=[\beta]_H=[\alpha]_H$.
\endproof

\section{A new topology on $\widetilde X$}\label{SECTION: BS topology}

We do not know how to characterize subgroups $H$ of $\pi_1(X,x_0)$
for which $p_H\colon \widehat X_H\to X$ has the unique path lifting property.
Therefore we will create a new topology on $\widetilde X_H$ for which
analogous question has a satisfactory answer in the case $H$ being a normal subgroup.

Given an open cover $\mathcal{U}$ of $X$, a subgroup $H$ of
$\pi_1(X,x_0)$, a path $\alpha$ in $X$
originating at $x_0$, and $V\in \mathcal{U}$ containing $x_1=\alpha(1)$
 define $B_H([\alpha]_H,\mathcal{U},V)\subset \widetilde X_H$ as follows: $[\beta]_H\in B_H([\alpha]_H,\mathcal{U},V)$
if and only if there is
a path $\gamma_0$ in $V$ originating at $x_1=\alpha(1)$
and a loop $\lambda$ at $x_1$ such that
$[\lambda]\in\pi(\mathcal{U},x_1)$ and
$\beta \sim_H \alpha \ast\lambda\ast\gamma_0$.
\par Observe $[\beta]_H\in B_H([\alpha]_H,\mathcal{U},V)$ implies
$B_H([\alpha]_H,\mathcal{U},V)=B_H([\beta]_H,\mathcal{U},V)$
 and
\par\noindent
$B_H([\alpha]_H,\mathcal{U}\cap\mathcal{V},V_1\cap V_2)\subset B_H([\alpha]_H,\mathcal{U},V_1)\cap B_H([\alpha]_H,\mathcal{V},V_2)$, so the family
of sets $\{B_H([\alpha]_H,\mathcal{U},V)\}$ forms a basis of a new topology
on $\widetilde X_H$. When we consider $\widetilde X_H$
as a topological space, then we use precisely that topology.
In the particular case of $H=\{1\}$, the trivial subgroup of $\pi_1(X,x_0)$,
we simplify $\widetilde X_H$ to $\widetilde X$.
Observe that, as $\pi_1(X,x_0)$ is the fiber of the endpoint
projection $p\colon \widetilde X\to X$, any subgroup $G$ of $\pi_1(X,x_0)$
can be considered as a subspace of $\widetilde X$ and we may consider it
as a topological space that way.

Notice the identity function $\widehat X_H\to\widetilde X_H$
is continuous. Indeed, $B_H([\alpha]_H,V)\subset B_H([\alpha]_H,\mathcal{U},V)$
for any $V\in \mathcal{U}$ containing $\alpha(1)$.

When dealing with the pointed topological category the space $\widetilde X_H$
is equipped with the base-point $\widetilde x_0$ equal to the equivalence class
of the constant path at $x_0$.

Let us prove a basic functorial property of our construction.

\begin{proposition}\label{FunctorialityOfWidetilde}
Suppose $f\colon (X,x_0)\to (Y,y_0)$ is a map of pointed topological
spaces. If $H$ and $G$ are subgroups of $\pi_1(X,x_0)$ and $\pi_1(Y,y_0)$,
respectively, such that $\pi_1(f)(H)\subset G$,
then $f$ induces a natural continuous function
$\tilde f\colon(\widetilde X_H,\widetilde x_0)\to(\widetilde Y_G,\widetilde y_0)$.
\end{proposition}
\proof Put $\tilde f([\alpha]_H)=[f\circ \alpha]_G$ and notice
$$\tilde f(B_H([\alpha]_H,f^{-1}(\mathcal{U}),f^{-1}(V)))\subset
B_G(\tilde f([\alpha]_H),\mathcal{U},V)$$ for any open covering
$\mathcal{U}$ of $Y$ and any neighborhood $V$ of $\alpha(1)$.
\endproof

In connection to \ref{BasicThmOnWidehatX} let us prove the following:

\begin{proposition}\label{DiscreteFibersOfWidetilde}
If $X$ is a path-connected space and $H$ is a subgroup of $\pi_1(X,x_0)$,
then the following conditions are equivalent:
\begin{itemize}
\item[a)] A fiber of the endpoint projection $p_H\colon \widetilde X_H\to X$
has an isolated point,

\item[b)] The endpoint projection $p_H\colon \widetilde X_H\to X$
has discrete fibers,
\item[c)] There is an open covering
$\mathcal{U}$ of $X$ so that $\pi(\mathcal{U},x_0)\subset H$,
\item[d)] $\widetilde X_H$ is a Peano space and $p_H\colon \widetilde X_H\to P(X)$
is a classical covering map.
\end{itemize}
\end{proposition}
\proof a)$\implies$c). Suppose $[\alpha]_H\in p_H^{-1}(x_1)$ is isolated. There is an open covering $\mathcal{U}$ of $X$ and $V\in \mathcal{U}$ containing $x_1$
 such that $B_H([\alpha]_H,\mathcal{U},V)\cap p_H^{-1}(x_1)=\{[\alpha]_H\}$.
Given $\gamma$ in $\pi(\mathcal{U},x_0)$, the homotopy class 
$[\alpha^{-1}\ast\gamma\ast\alpha]_H$ belongs to $\pi(\mathcal{U},x_1)$, so
$[\alpha\ast \alpha^{-1}\ast\gamma\ast\alpha]_H=[\gamma\ast\alpha]_H$
belongs to
$B_H([\alpha]_H,\mathcal{U},V)\cap p_H^{-1}(x_1)$.
Hence $[\gamma\ast\alpha]_H=[\alpha]_H$ and
$[\gamma]\in H$.
\par c)$\implies$d).
Suppose there is an open covering
$\mathcal{U}$ of $X$ so that $\pi(\mathcal{U},x_0)\subset H$
and $W$ is a path component of $U\in\mathcal{U}$.
Notice $B_H([\alpha]_H,\mathcal{U},U)$ is mapped by $p_H$
bijectively onto $W$ and that is sufficient for d).

d)$\implies$b) and b)$\implies$a) are obvious.
\endproof

Applying \ref{DiscreteFibersOfWidetilde} to $H$ being trivial one gets
the following (see \cite{Fab} for analogous result in case of
a different topology on the fundamental group):
\begin{corollary}\label{DiscretePiOne}
If $X$ is a path-connected space, then $\pi_1(X,x_0)$
is discrete if and only if $X$ is semilocally simply connected.
\end{corollary}

\begin{proposition}\label{ComparisonOfBSAndOursForHTotal}
If $\pi(\mathcal{V},x_0)\subset H$ for some open cover
$\mathcal{V}$ of $X$, then the identity function $\widehat X_H\to \widetilde X_H$
is a homeomorphism.
\end{proposition}
\proof Let us show $B_H([\alpha]_H,\mathcal{U},W)=B_H([\alpha]_H,W)$
if $\mathcal{U}$ is an open cover of $X$ refining $\mathcal{V}$ and
$W$ is an element of $\mathcal{U}$ containing $\alpha(1)$.
Clearly, $B_H([\alpha]_H,W)\subset B_H([\alpha]_H,\mathcal{U},W)$,
so assume $[\beta]_H\in B_H([\alpha]_H,\mathcal{U},W)$.
There are $h\in H$, $[\lambda]\in\pi(\mathcal{U},\alpha(1))$,
and a path $\gamma$ in $W$ such that $[\beta]=[h\ast\alpha\ast\lambda\ast\gamma]$.
Choose $h_1\in H$ so that $[h_1\ast\alpha]=[\alpha\ast\lambda]$
($h_1=[\alpha\ast \lambda\ast\alpha^{-1}]\in\pi(\mathcal{U},x_0)\subset H$).
Now $[\beta]=[h\ast\alpha\ast\lambda\ast\gamma]=[h\ast h_1\ast\alpha\ast\gamma]$
and $[\beta]_H\in B_H([\alpha]_H,W)$.
\par Now we can show the identity function $\widehat X_H\to \widetilde X_H$
is open: given an open cover $\mathcal{W}$ of $X$ and given
a path $\alpha$ from $x_0$ to $x_1$ pick an element $W$ of $\mathcal{U}=\mathcal{W}\cap\mathcal{V}$ containing $x_1$ and notice $B_H([\alpha]_H,\mathcal{U},W)\subset B_H([\alpha]_H,W)$.
\endproof

\begin{lemma}\label{ProjectionIsOpenLemma}
If $G\subset H$ are subgroups of $\pi_1(X,x_0)$, then the projection $p\colon \widetilde X_G\to \widetilde X_H$ is open.
\end{lemma}
\proof It suffices to show
$p(B_G([\alpha]_G,\mathcal{U},V))=B_H([\alpha]_H,\mathcal{U},V)$.
Clearly,
\par\noindent
$p(B_G([\alpha]_G,\mathcal{U},V))\subset B_H([\alpha]_H,\mathcal{U},V)$,
so suppose $[\beta]_H\in B_H([\alpha]_H,\mathcal{U},V)$
and $[\beta]=[h\ast \alpha\ast \lambda\ast\gamma]$,
where $[\lambda]\in\pi(\mathcal{U},\alpha(1))$ and $\gamma$
is a path in $V$ originating at $\beta(1)$. Observe
$[\beta]_H=[\alpha\ast\lambda\ast\gamma]_H=p([\alpha\ast\lambda\ast\gamma]_G)$.
\endproof

We arrived at the fundamental result for the new topology on $\widetilde X_H$:

\begin{theorem}\label{TopActionTheorem}
Suppose $G\subset H$ are subgroups of $\pi_1(X,x_0)$.
If $G$ is normal in $\pi_1(X,x_0)$, then $H/G$, identified with the fiber $p^{-1}([\tilde x_0]_H)$
of the projection $p\colon \widetilde X_G\to \widetilde X_H$,
is a topological group and
acts continuously on $\widetilde X_G$
so that
\begin{itemize}
\item[a)] The natural map $(H/G)\times \widetilde X_G\to  \widetilde X_G\times \widetilde X_G$ defined by $([\alpha]_G,[\beta]_G)\mapsto ([\alpha*\beta]_G,[\beta]_G)$ is an embedding,
\item[b)] The quotient map from $\widetilde X_G$ to the orbit space corresponds
to the projection $p\colon\widetilde X_G\to \widetilde X_H$.
\end{itemize}
\end{theorem}
\proof The fiber $F$
of the projection $p\colon \widetilde X_G\to \widetilde X_H$
is the set of classes $[\alpha]_G$ such that $[\alpha]\in H$, so it corresponds
to $H/G$. Define $\mu\colon F\times \widetilde X_G\to \widetilde X_G$ as follows: given $[\alpha]_G\in F$ and given $[\beta]_G\in \widetilde X_G$
put $\mu([\alpha]_G,[\beta]_G)= [\alpha\ast\beta]_G$.
To see $\mu$ is well defined assume $[\gamma_1], [\gamma_2]\in G$.
Now $[\gamma_1\ast \alpha\ast\gamma_2\ast\beta]_G[(\alpha\ast\gamma_2\ast\alpha^{-1})\ast (\alpha\ast\beta)]_G=[\alpha\ast\beta]_G$
as $[\alpha\ast\gamma_2\ast\alpha^{-1}]\in G$ due to normality of $G$ in $H$.

\par Suppose $\mathcal{U}$ is an open cover of $X$, $V,V_1\in\mathcal{U}$, and
\begin{enumerate}
\item $[\alpha]_G\in F$, $[\beta]_G\in \widetilde X_G$,
\item $[\alpha_1]_G\in B_G([\alpha]_G,\mathcal{U},V_1)\cap F$,
and $[\beta_1]_G\in B_G([\beta]_G,\mathcal{U},V)$.

\end{enumerate}

Thus $[\alpha_1]=[g_1\ast \alpha\ast \lambda_1]$ for some
$[\lambda_1]\in\pi(\mathcal{U},x_0)$ and $[g_1]\in G$.
Similarly, $[\beta_1]= [g_2\ast \beta\ast\lambda_2\ast\gamma]$,
where $[g_2]\in G$, $[\lambda_2]\in\pi(\mathcal{U},\beta(1))$,
and $\gamma$ is a path in $V$.
Now,
$$[\alpha_1^{-1}\ast\beta_1]_G=[\lambda_1^{-1}\ast \alpha^{-1}\ast g_1^{-1}\ast g_2\ast\beta\ast\lambda_2\ast\gamma]_G=$$

$$[(\lambda_1^{-1}\ast \alpha^{-1}\ast g_1^{-1}\ast g_2\ast\alpha\ast\lambda_1)\ast\lambda_1^{-1}\ast \alpha^{-1}\ast\beta\ast\lambda_2\ast\gamma]_G=$$

$$[\lambda_1^{-1}\ast \alpha^{-1}\ast\beta\ast\lambda_2\ast\gamma]_G=[(\alpha^{-1}\ast\beta)\ast (\beta^{-1}\ast\alpha\ast \lambda_1^{-1}\ast \alpha^{-1}\ast\beta)\ast\lambda_2\ast\gamma]_G \in
B_G([\alpha^{-1}\ast\beta]_G,\mathcal{U},V)$$
 as $[\lambda_1^{-1}\ast \alpha^{-1}\ast g_1^{-1}\ast g_2\ast\alpha\ast\lambda_1]\in G$
 and
$[\beta^{-1}\ast\alpha\ast \lambda_1^{-1}\ast \alpha^{-1}\ast\beta]\in \pi(\mathcal{U},(\alpha^{-1}\ast\beta)(1))$.
\par The above calculations amount to
$$\rho((F\cap B_G(x,\mathcal{U},V_1))\times B_G(y,\mathcal{U},V))\subset B_G(\rho(x,y),\mathcal{U},V),$$
where $\rho(x,y):=\mu(x^{-1}, y)$,
 which implies the following

 \begin{enumerate}
\item $F$ is a topological group,
\item $\mu$  is continuous,
\item  $(x,y)\to (\mu(x^{-1},y),y)$
from $F\times \widetilde X_G$ onto its image is open.
\end{enumerate}
 As the map in (3) is injective,
it is an embedding. Hence $(x,y)\to (\mu(x,y),y)$ is an embedding.
\par To see b) use \ref{ProjectionIsOpenLemma} or check it directly.
\endproof

\section{Path lifting}\label{SECTION: PathLifting}

\begin{definition}\label{PathLiftingDef}
A pointed map $f\colon (X,x_0)\to (Y,y_0)$ has the {\bf path lifting property}
if any path $\alpha\colon (I,0)\to (Y,y_0)$ has a lift $\beta\colon (I,0)\to (X,x_0)$.
\par A surjective map $f\colon X\to Y$ has the {\bf path lifting property}
if for any path $\alpha\colon I\to Y$ and any $y_0\in f^{-1}(\alpha(0))$ there is a lift $\beta\colon I\to X$ of $\alpha$ such that $\beta(0)=y_0$.
\end{definition}

\begin{definition}\label{UniquenessOfPathLiftsDef}
A pointed map $f\colon (X,x_0)\to (Y,y_0)$ has the {\bf uniqueness of path lifts property}
if any two paths $\alpha,\beta\colon (I,0)\to (X,x_0)$ are equal if $f\circ\alpha=f\circ \beta$.
\par A pointed map $f\colon (X,x_0)\to (Y,y_0)$ has the {\bf unique path lifting property}
if it has both the path lifting property and the uniqueness of path lifts property.
\par A map $f\colon X\to Y$ has the {\bf uniqueness of path lifts property}
if any two paths $\alpha,\beta\colon I\to X$ are equal if $f\circ\alpha=f\circ \beta$
and $\alpha(0)=\beta(0)$.
\par  A surjective map $f\colon X\to Y$ has the {\bf unique path lifting property}
if it has both the path lifting property and the uniqueness of path lifts property.
\end{definition}

\begin{corollary}\label{UniquePathLiftingForHandG}
Supppose $G\subset H$ are subgroups of $\pi_1(X,x_0)$.
If $G$ is normal in $\pi_1(X,x_0)$, then the following conditions
are equivalent:
\begin{itemize}
\item[a)]  The natural map $\widetilde X_G\to \widetilde X_H$
has the uniqueness of path lifts property,
\item[b)] $\pi_0(H/G)=H/G$, i.e. $H/G$ has trivial path components.
\end{itemize}
\end{corollary}
\proof a)$\implies$b). If $H/G$ has a non-trivial path component, then there is a non-trivial lift of the constant path at the base-point of $\widetilde X_H$.
\par b)$\implies$a). Suppose $\alpha$ and $\beta$ are two lifts of
the same path $\gamma$ in $\widetilde X_H$ and $\alpha(0)=\beta(0)$. By
\ref{TopActionTheorem} there is a path
$\lambda$ in $H/G$ with the property $\lambda(t)\cdot \alpha(t)=\beta(t)$
for each $t\in I$. As $\lambda(0)=1\in H/G$ and $H/G$ has trivial path components,
$\lambda(t)=1\in H/G$ for all $t\in I$ and $\alpha=\beta$.
\endproof

\begin{proposition}\label{FibersAreT2Lemma}
Supppose $G\subset H$ are subgroups of $\pi_1(X,x_0)$.
If $G$ is normal in $\pi_1(X,x_0)$, then the following conditions are equivalent:
\begin{itemize}
\item[a)] $H/G$ is a $T_0$-space,
\item[b)] $H/G$ is Hausdorff,

\item[c)] Fibers of the projection $p\colon \widetilde X_G\to \widetilde X_H$ are $T_0$,
\item[d)] Fibers of the projection $p\colon \widetilde X_G\to \widetilde X_H$ are Hausdorff,
\item[e)] For each $h\in H-G$ there is
a cover $\mathcal{U}$ such that $(G\cdot h)\cap \pi(\mathcal{U},x_0)=\emptyset$,
\item[f)] $G$ is closed in $H$.
\end{itemize}
\end{proposition}
\proof In view of \ref{TopActionTheorem}, a)$\equiv$c) and b)$\equiv$d).
\par

a)$\implies$e). Assume $H/G$ is $T_0$ and $h\in H- G$.
Since $[\beta]_G\in B_G([\alpha]_G,\mathcal{U},V)$
is equivalent to $[\alpha]_G\in B_G([\beta]_G,\mathcal{U},V)$,
there is an open cover $\mathcal{U}$ and $V\in\mathcal{U}$ containing $x_0$
 such that $G\cdot h\notin B_G(G\cdot 1,\mathcal{U},V)$.
That means precisely there is no $\lambda\in\pi(\mathcal{U},x_0)$
such that $G\cdot h=G\cdot \lambda$, hence $(G\cdot h)\cap \pi(\mathcal{U},x_0)=\emptyset$.
\par b)$\equiv$d) and a)$\equiv$c) follow from  \ref{TopActionTheorem}.
\par e)$\implies$d).
Suppose $\alpha,\beta$ are two paths in $(X,x_0)$
so that $[\alpha]_H=[\beta]_H$ but $[\alpha]_G\ne [\beta]_G$.
choose $h\in H- G$ satisfying $[h\cdot \alpha]=[\beta]$.
Pick an open cover $\mathcal{U}$ of $X$ satisfying
$G\cdot h\cap \pi(\mathcal{U},x_0)=\emptyset$ and let $V\in\mathcal{U}$ contain $\alpha(1)$.
Suppose $[\gamma]_G\in B_G([\alpha]_G,\mathcal{U},V)\cap B_G([\beta]_G,\mathcal{U},V)$
and $[\gamma]_H=[\alpha]_H$. Let $h_0\in H$ satisfy $[h_0\cdot \alpha]=[\gamma]$.
Choose $\lambda_1,\lambda_2\in\pi(\mathcal{U},\alpha(1))$
such that $G\cdot [h_0\cdot \alpha]=G\cdot \alpha\cdot \lambda_1$ and
$G\cdot [h_0\cdot \alpha]=G\cdot [h\cdot \alpha]\cdot \lambda_2$.
As $G$ is normal in $H$,
$G\cdot h=h\cdot G=G\cdot (\alpha\cdot \lambda_1\cdot\lambda_2^{-1}\alpha^{-1})$,
a contradiction as $\alpha\cdot \lambda_1\cdot\lambda_2^{-1}\cdot \alpha^{-1}\in \pi(\mathcal{U},x_0)$.

b)$\implies$a) is obvious.

e)$\equiv$f). $G$ being closed in $H$ means existence, for each $h\in H- G$, of an open cover $\mathcal{U}$ such that $G\cap B(h,\mathcal{U},V)=\emptyset$
for some $V\in\mathcal{U}$ containing $x_0$.
That, in turn, is equivalent to non-existence of $\lambda\in\pi(\mathcal{U},x_0)$
satisfying $h\cdot\lambda\in G$, i.e. $(G\cdot h^{-1})\cap \pi(\mathcal{U},x_0)=\emptyset$.
\endproof

 \begin{corollary}\label{GroupIsClosed}
 Suppose $G\subset H$ are subgroups of $\pi_1(X,x_0)$.
If $G$ is a normal subgroup of $\pi_1(X,x_0)$, then
the following conditions are equivalent:
\begin{itemize}
\item[a.] $H/G$ has trivial components,
\item[b.] $H/G$ has trivial path components,
\item[c.] $G$ is closed in $H$.
\end{itemize}
 \end{corollary}
 \proof b)$\implies$c). Suppose $H/G$ has trivial path components. In view of
 \ref{FibersAreT2Lemma} it suffices to show $H/G$ is $T_0$
 to deduce $G$ is closed in $H$.
 If there are two points $u$ and  $v$ of $H/G$ such that any open subset
 of $H/G$ either contains both of them or contains none of them,
 then any function $I\to \{u,v\}\subset H/G$ is continuous. Hence $u=v$
 as $H/G$ does not contain non-trivial paths.
 \par c)$\implies$a). \par
 {\bf Claim.} If $h_1,h_2\in H$ and $G\cdot f\in B_H(G\cdot h_1,\mathcal{U},V)\cap B_H(G\cdot h_{2},\mathcal{U},V)\cap (H/G)$ for some open cover $\mathcal{U}$ of $X$
 and some $V\in\mathcal{U}$ containing $x_0$, then
$G\cdot h_1^{-1}\cdot h_2\subset G\pi(\mathcal{U},x_0)$.
\par {\bf Proof of Claim:}
 $G\cdot f=G\cdot h_1\cdot \lambda_1$ and $G\cdot f=G\cdot h_2\cdot \lambda_2$ for some $\lambda_1,\lambda_2\in\pi(\mathcal{U},x_0)$.
Now $h_1\cdot G=h_2\cdot G\cdot (\lambda_2\cdot \lambda_1^{-1})$
and $(h_1^{-1}\cdot h_2)\cdot G\subset G\cdot (\lambda_1\cdot \lambda_2^{-1})
\subset G\pi(\mathcal{U},x_0)$.
\endproof

Suppose $G$ is closed in $H$ and $h\in H- G$.
 By \ref{FibersAreT2Lemma} there is
a cover $\mathcal{U}$ such that $(G\cdot h)\cap \pi(\mathcal{U},x_0)=\emptyset$.
If there is a connected subset $C$ of $H/G$ containing $G\cdot h_1h$ and $G\cdot h_1$ for some $h_1\in H$, we consider the open cover $\{C\cap B_G(z,\mathcal{U},V)\}_{z\in C}$
of $C$ and define the equivalence relation on $C$ determined by that cover
($z\sim z^\prime$ if  there is a finite chain $z=z_1,\ldots,z_k=z^\prime$ in $C$
such that $B_G(z_i,\mathcal{U},V)\cap B_G(z_{i+1},\mathcal{U},V)\cap C\ne\emptyset$
for all $i<k$). Equivalence classes of that relation are open, hence closed and must equal $C$.
Thus there is a finite chain $h_1,\ldots,h_k=h_1\cdot h$ in $H$
such that $B_G([h_i]_G,\mathcal{U},V)\cap B_G([h_{i+1}]_G,\mathcal{U},V)\cap (H/G)\ne\emptyset$
for all $i<k$. By Claim there are elements $g_i\in G$ ($i < k$) so that
$g_i\cdot h_i^{-1}\cdot h_{i+1}\in \pi(\mathcal{U},x_0)$. By normality of $G$ in $H$
there is $g\in G$ satisfying $g\cdot \prod\limits_{i=1}^{k-1}h_i^{-1}\cdot h_{i+1}
=g\cdot h\in \pi(\mathcal{U},x_0)$, a contradiction.
 \endproof

\begin{theorem}\label{PiOneOfXHAndPathLifting}
If $G$ is a normal subgroup of $\pi_1(X,x_0)$,
then the following conditions are equivalent:
\begin{itemize}
\item[a.] The endpoint projection
$p_G\colon(\widetilde X_G,\widetilde x_0)\to (X,x_0)$ has the unique
path lifting property,
\item[b.] $G$ is closed in $\pi_1(X,x_0)$,
\item[c.]  $\pi_1(p_G)\colon \pi_1(\widetilde X_G,\widetilde x_0)\to\pi_1(X,x_0)$
is a monomorphism and its image equals $G$.
\end{itemize}
\end{theorem}
\proof Put $H=\pi_1(X,x_0)$ and observe $\widetilde X_H$
is the Peanification of $(X,x_0)$ by \ref{CaseHBeingWholeGroup}.

a)$\equiv$b). By \ref{UniquePathLiftingForHandG} the group $H/G$ has trivial path components. Use \ref{GroupIsClosed}.

a)$\implies$c). Given a loop in $(\widetilde X_G,\widetilde x_0)$ we may assume
it is a canonical lift of a loop $\alpha$ in $(X,x_0)$.
For that lift to be a loop we must have $[\alpha]\in G$.
Thus the image of $\pi_1(p_G)\colon \pi_1(\widetilde X_G,\widetilde x_0)\to\pi_1(X,x_0)$
equals $G$ (canonical lifts of elements of $G$ show that the image
contains $G$). If $\alpha$ is null-homotopic in $(X,x_0)$,
then its canonical lift is null-homotopic as well.
Thus $\pi_1(p_G)\colon \pi_1(\widetilde X_G,\widetilde x_0)\to\pi_1(X,x_0)$
is a monomorphism.

c)$\implies$a). If $H/G$ has a non-trivial path component (we use~\ref{UniquePathLiftingForHandG}),
then there is a path from the base-point to a different point $[\alpha]_G$
of $H/G$. Concatenating the canonical lift of $\alpha$ with the reverse of that path gives
a loop in $(\widetilde X_G,\widetilde x_0)$ whose image in $\pi_1(X,x_0)$ is $[\alpha]\notin G$,
a contradiction.
\endproof

\begin{proposition}\label{ClosureOfSubgroups}
Suppose $(X,x_0)$ is a pointed topological space and $H$ is a subgroup of $\pi_1(X,x_0)$.
The closure of $H$ in  $\pi_1(X,x_0)$
consists of all elements $g\in \pi_1(X,x_0)$ such that for each open
cover $\mathcal{U}$ of $X$ there is $h\in H$ and $\lambda\in\pi(\mathcal{U},x_0)$
satisfying $g=h\cdot \lambda$. If $H$ is a normal subgroup of $\pi_1(X,x_0)$,
then so is its closure.
\end{proposition}
\proof Suppose $g\in \pi_1(X,x_0)$ and for each open
cover $\mathcal{U}$ of $X$ there is $h\in H$ and $\lambda\in\pi(\mathcal{U},x_0)$
satisfying $g=h\cdot \lambda$. Notice $B(g,\mathcal{U})$ contains $h$, so $g$
belongs to the closure of $H$. If $H$ is normal, then
$k\cdot g\cdot k^{-1}=(k\cdot h\cdot k^{-1})\cdot (k\cdot \lambda\cdot k^{-1})$ also belongs to the closure of $H$.
\endproof

\begin{corollary}\label{ExamplesOfClosedSubgroupsOne}
The closure of the trivial subgroup of $\pi_1(X,x_0)$ in  $\pi_1(X,x_0)$ equals
$\bigcap\limits_{\mathcal{U}\in COV}\pi(\mathcal{U},x_0)$, where $COV$ stands
for the family of all open covers of $X$.
\end{corollary}

\begin{example}\label{HarmonicArchipelago}
The Harmonic Archipelago $HA$ of Bogley and Sieradski \cite{BogSie}
is a Peano space such that $\pi_1(X,x_0)$ equals
$\bigcap\limits_{\mathcal{U}\in COV}\pi(\mathcal{U},x_0)$. Hence $\pi_1(X,x_0)$
is the only closed subgroup of $\pi_1(X,x_0)$.
$HA$ is built by stretching disks $B(2^{-n},2^{-n-2})$ to form cones
over its boundary with the vertices at height $1$ in the $3$-space.
\end{example}

\begin{corollary}\label{ExamplesOfClosedSubgroupsTwo}
Suppose $(X,x_0)$ is a pointed topological space.
The following subgroups of $\pi_1(X,x_0)$ are closed:
\begin{itemize}
\item[a)] Subgroups $H$ containing $\pi(\mathcal{U},x_0)$ for some open cover $\mathcal{U}$ of $X$,
\item[b)] $\bigcap\limits_{\mathcal{U}\in S}\pi(\mathcal{U},x_0)$
for any family $S$ of open covers of $X$,
\item[c)] The kernel of $\pi_1(f)\colon \pi_1(X,x_0)\to\pi_1(Y,y_0)$
for any map $f\colon (X,x_0)\to (Y,y_0)$ to a pointed semilocally simply connected space.
\item[d)] The kernel of the natural homomorphism $\pi_1(X,x_0)\to\check\pi_1(X,x_0)$
from the fundamental group to the \v Cech fundamental group.
\end{itemize}
\end{corollary}
\proof a) Any subgroup containing $\pi(\mathcal{U},x_0)$ is open. Any open subgroup of a topological group is closed.
\par b) easily follows from a).
\par c) follows from \ref{DiscretePiOne} and \ref{FunctorialityOfWidetilde}
as $\pi_1(f)\colon \pi_1(X,x_0)\to\pi_1(Y,y_0)$ is continuous and
$\pi_1(Y,y_0)$ is discrete.
\par d) follows from c). Indeed $\check\pi_1(X,x_0)$ is defined (see \cite{DydSeg}
or \cite{MarSeg}) as the inverse limit of an inverse system $\{\pi_1(K_s,k_s)\}_{s\in S}$,
where each $K_s$ is a simplicial complex and there are maps
$f_s\colon (X,x_0)\to (K_s,k_s)$ so that for $t > s$ the map $f_s$ is homotopic
to the composition of $f_t$ and the bonding map $(K_t,k_t)\to (K_s,k_s)$.
That means the kernel of the natural homomorphism $\pi_1(X,x_0)\to\check\pi_1(X,x_0)$
is the intersection of kernels of all $\pi_1(f_s)$, $s\in S$.
\endproof

The concept of a space $X$ being {\bf homotopically Hausdorff} was introduced
by Conner and Lamoreaux \cite[Definition 1.1]{ConLam} 
to mean that for any point $x_0$ in $X$ and for any non-homotopically trivial
loop $\gamma$ at $x_0$ there is a neighborhood $U$ of $x_0$ in $X$
with the property that no loop in $U$ is homotopic to $\gamma$ rel.$x_0$
in $X$. Subsequently, Fischer and Zastrow \cite{FisZas}) defined a space $X$ to be {\bf homotopically Hausdorff relative to
a subgroup $H$ of $\pi_1(X,x_0)$} if for any $g\notin H$
and for any path $\alpha$ originating at $x_0$ there is an open neighborhood $U$
of $\alpha(1)$ in $X$ such that no element of $H\cdot g$ can be expressed as $[\alpha\ast\gamma\ast\alpha^{-1}]$ for some loop $\gamma$ in $(U,\alpha(1))$.
We generalize this definition as follows:
\begin{definition}\label{HGHausdorffDef}
Suppose $G\subset H$ are subgroups of $\pi_1(X,x_0)$.
$X$ is {\bf $(H,G)$-homotopically Hausdorff}
if for any $h\in H\setminus G$ and any path $\alpha$
originating at $x_0$ there is an open neighborhood $U$
of $\alpha(1)$ in $X$ such that none of the elements of $G\cdot h$ can be expressed as $[\alpha\ast\gamma\ast\alpha^{-1}]$ for any loop $\gamma$ in $(U,\alpha(1))$.
\end{definition}

Notice $X$ being homotopically Hausdorff relative to $H$ corresponds
to $X$ being $(\pi_1(X,x_0),H)$-homotopically Hausdorff.

Let us characterize the concept of being $(H,G)$-homotopically Hausdorff
in terms of the basic topology on the fundamental group.

\begin{proposition}\label{HGHausdorffAndBasicTopology}
If $G\subset H$ are subgroups of $\pi_1(X,x_0)$,
then $X$ is $(H,G)$-homotopically Hausdorff if and only if for every path
$\alpha$ in $X$ that terminates at $x_0$ the group
$h_\alpha(G)$ is closed in $h_\alpha(H)$ in the basic topology.
\end{proposition}
\proof $h_\alpha(G)$ being closed in $h_\alpha(H)$ means existence,
for each $h\in H\setminus G$, of a neighborhood U of $x_1=\alpha(0)$
such that $B([\alpha\ast h\ast\alpha^{-1}],U)\cap ([\alpha]\cdot G\cdot [\alpha^{-1}])=\emptyset$. Thus, for every loop $\gamma$ in $U$ at $x_1$, there is no $g\in G$
satisfying $[\alpha\ast h\ast\alpha^{-1}\ast\gamma^{-1}]=[\alpha\ast g\ast\alpha^{-1}]$.
The last equality is equivalent to $[g\ast h]=[\alpha^{-1}\ast\gamma\ast\alpha]$
which completes the proof.
\endproof

\begin{example} Proposition \ref{HGHausdorffAndBasicTopology}
allows for an easy construction of subgroups $H$ of $\pi_1(X,x_0)$
such that $X$ is not homotopically Hausdorff relative to $H$.
Namely, $X=S^1\times S^1\times \ldots$ and $H=\bigoplus Z\subset
\prod Z=\pi_1(X)$.
\end{example}

Let us show $G$ being closed in $H$ (in the new topology) is a stronger condition
than $X$ being $(H,G)$-homotopically Hausdorff.

\begin{lemma}\label{ClosedImpliesHH}
Suppose $G\subset H$ are subgroups of $\pi_1(X,x_0)$.
If $G$ is closed in $H$, then  $X$ is $(H,G)$-homotopically Hausdorff.
\end{lemma}
\proof Given $h\in H\setminus G$ pick an open cover $\mathcal{U}$
and $W\in\mathcal{U}$ containing $x_0$
so that $B(h,\mathcal{U},W)$ does not intersect $G$.
Given a path $\alpha$ in $X$ from $x_0$ to $x_1$ choose $V\in\mathcal{U}$
containing $x_1$. Suppose there is a loop $\gamma$ in $(V,x_1)$
so that $[\alpha\ast\gamma\ast\alpha^{-1}]=g\cdot h$ for some $g\in G$.
Now $[\alpha\ast\gamma^{-1}\ast\alpha^{-1}]\in \pi(\mathcal{U},x_0)$
and $g^{-1}=h\ast [\alpha\ast\gamma^{-1}\ast\alpha^{-1}]\in G\cap B(h,\mathcal{U},W)$, a contradiction.
\endproof

\begin{remark}\label{StronglyHHRemark}
The proof of \ref{ClosedImpliesHH} suggests that the trivial subgroup of
$\pi_1(X,x_0)$ being closed is philosophically related
to the concept of $X$ being {\bf strongly homotopically Hausdorff} (see \cite{RepZas}).
Recall a metric space $X$ is strongly homotopically Hausdorff if
 for any non-null-homotopic loop $\alpha$ in $X$ there is an $\epsilon > 0$
 such that $\alpha$ is not freely homotopic to a loop of
 diameter less than $\epsilon$.
\end{remark}

\begin{lemma}\label{HSmallLemma}
Given subgroups $G\subset H$ of $\pi_1(X,x_0)$ the following conditions are equivalent:
\begin{itemize}
\item[a)] The fibers of the natural projection
$p\colon\widehat X_G\to \widehat X_H$ are $T_0$,
\item[b)] The fibers of the natural projection
$p\colon\widehat X_G\to \widehat X_H$ are Hausdorff,
\item[c)] $X$ is $(H,G)$-homotopically Hausdorff.
\end{itemize}
\end{lemma}
\proof a)$\implies$c). Suppose $h\in H\setminus G$ and $\alpha$
is a path in $X$ from $x_0$ to $x_1$.
As $[h\ast\alpha]_G\ne [\alpha]_G$ belong to the same fiber
of $p$, there is a neighborhood $U$ of $x_1$ so that $[h\ast\alpha]_G\notin
B_G([\alpha]_G,U)$ or $[\alpha]_G\notin
B_G([h\ast\alpha]_G,U)$. Notice $[h\ast\alpha]_G\notin
B_G([\alpha]_G,U)$ is equivalent to $[\alpha]_G\notin
B_G([h\ast\alpha]_G,U)$.
Suppose there is a loop $\gamma$ in $(U,x_1)$
so that $g\cdot h=[\alpha\ast\gamma\ast\alpha^{-1}]$ for some $g\in G$.
Now $[h\ast\alpha]_G=[g\cdot h\ast\alpha]_G=[\alpha\ast\gamma]_G\in B_G([\alpha]_G,U)$,
a contradiction.

c)$\implies$b). Any two different elements of the same fiber of $p$ can be represented
as $[h\ast\alpha]_G\ne [\alpha]_G$ for some path $\alpha$
 in $X$ from $x_0$ to $x_1$ and some $h\in H\setminus G$.
Choose a neighborhood $U$ of $x_1$ with the property
that none of the elements of $G\cdot h$ can be expressed as $[\alpha\ast\gamma\ast\alpha^{-1}]$ for any loop $\gamma$ in $(U,x_1)$.
Suppose $[\beta]_G\in (H/G)\cap B_G([\alpha]_G,U)\cap B_G([h\ast\alpha]_G,U)$.
That means existence of loops $\gamma_1,\gamma_2$
in $(U,x_1)$ so that $[\beta]_G=[h\ast\alpha\ast\gamma_1]_G=[\alpha\ast\gamma_2]_G$.
Hence $[h]_G=[\alpha\ast(\gamma_2\ast\gamma_1^{-1})\ast \alpha^{-1}]_G$,
a contradiction.
\endproof

\begin{lemma}\label{NewHMediumLemma}
Supppose $G\subset H$ are subgroups of $\pi_1(X,x_0)$, $G$ is normal in $\pi_1(X,x_0)$,
and $X$ is $(H,G)$-homotopically Hausdorff.
If $\alpha,\beta\colon (I,0)\to (\widehat X_G,\widehat x_0)$ are two
continuous lifts of the same path
$\gamma\colon (I,0)\to (\widehat X_H,\widehat x_0)$, then for every $h\in H$ the set
$$S=\{t\in I | \alpha(t)= h\cdot \beta(t)\}$$ is closed.
\end{lemma}
\proof
Choose paths $u_t,v_t$ in $(X,x_0)$
so that $\alpha(t)=[u_t]_G$ and $\beta(t)=[v_t]_G$ for all $t\in I$.
Assume $[u_t]_G\ne [h\cdot v_t]_G$ for some $t\in I$.
Pick a neighborhood $U$ of $x_1=u_t(1)$ so that
$[v_t\ast u_t^{-1}]\cdot h\cdot G\ne [v_t\ast \gamma\ast v_t^{-1}]\cdot G$
for any loop $\gamma$ in $(U,x_1)$.
There is a neighborhood $V$ of $t$ in $I$ so that
$[u_s]_G\in B_G([u_t]_G,U)$ and $[v_s]_G\in B_G([v_t]_G,U)$
for all $s\in V$. That means $[u_s]=[g_1\ast u_t\ast\gamma_1]$
and $[v_s]=[g_2\ast v_t\ast \gamma_2]$
for some $g_1,g_2\in G$ and some paths $\gamma_1,\gamma_2$
in $U$ joining $x_1$ and $u_1(1)=v_s(1)$.
Put $\gamma=\gamma_1\ast \gamma_2^{-1}$
and notice $[u_s\ast v_s^{-1}]=[g_1\ast u_t\ast v_t^{-1}\ast (v_t\ast \gamma\ast v_t^{-1})\ast g_2^{-1}]$.
As $G$ is normal in $\pi_1(X,x_0)$, there is $g_3\in G$
satisfying
$[g_1\ast u_t\ast v_t^{-1}\ast (v_t\ast \gamma\ast v_t^{-1})\ast g_2^{-1}]=[g_3\ast u_t\ast v_t^{-1}\ast (v_t\ast \gamma\ast v_t^{-1})]$
and that element cannot belong to $G\cdot h$ by the choice of $U$.
\endproof

\begin{corollary}\label{HAndPathLifting}
Supppose $G\subset H$ are subgroups of $\pi_1(X,x_0)$.
If $H/G$ is countable, $G$ is normal in $\pi_1(X,x_0)$,
and $X$ is $(H,G)$-homotopically Hausdorff, then the natural map
$\widehat X_G\to \widehat X_H$ has the uniqueness of path lifts property.
\end{corollary}
\proof Pick representatives $h_i\in H$, $i\ge 1$, of all right cosets of $H/G$
so that $h_1=1$.
If $\alpha$ and $\beta$ are two continuous lifts
in $\widehat X_G$ of the same path in $\widehat X_H$,
then each set $S_i= \{t\in I | \alpha(t)= h_i\cdot \beta(t)\}$ is closed,
they are disjoint, and their union is the whole interval $I$.
Hence only one of them is non-empty and it must be $S_1$.
Thus $\alpha=\beta$.
\endproof

\section{Peano maps}\label{SECTION Peano-maps}

This section is about one of the main ingredients of our theory
of covering maps for lpc-spaces. It amounts to the following generalization of Peano spaces:

\begin{definition}\label{PeanoMapDef}
A map $f\colon X\to Y$ is a {\bf Peano map} if the family of path components
of $f^{-1}(U)$, $U$ open in $Y$, forms a basis of neighborhoods of $X$.
\end{definition}

Notice $X$ is an lpc-space if $f\colon X\to Y$ is a Peano map.
One may reword the above definition as follows:
$X$ is an lpc-space and lifts of short paths in $Y$ are short in $X$.
Indeed, given a neighborhood $U$ of $x_0\in X$ there is a neighborhood $V$
of $f(x_0)$ in $Y$ such that any path $\alpha$ in $(f^{-1}(V),x_0)$
(i.e. $f\circ \alpha$ is contained in $V$, hence short) must be contained in $U$.

\begin{proposition}\label{ProductOfPeanoMaps}
Any product of Peano maps is a Peano map.
\end{proposition}
\proof Suppose $f_s\colon X_s\to Y_S$, $s\in S$, are Peano maps.
Observe $X=\prod\limits_{s\in S}X_s$ is an lpc-space.
Given a neighborhood $U$ of $x=\{x_s\}_{s\in S}\in X$, we find a finite subset $T$ of $S$
and neighborhoods $U_s$ of $x_s$ in $X_s$ such that
$\prod\limits_{s\in S}U_s\subset U$ and $U_s=X_s$ for $s\notin T$.
Choose neighborhoods $V_s$ of $f_s(x_s)$ in $Y_s$, $s\in T$,
so that the path-component of $x_s$ in $f_s^{-1}(V_s)$ is contained in $U_s$.
Put $V_s=X_s$ for $s\notin T$ and observe the path component
of $x$ in $f^{-1}(V)$, $f=\prod\limits_{s\in S}f_s$ and $V=\prod\limits_{s\in S}V_s$,
is contained in $U$.
\endproof

Here is our basic class of Peano maps:

\begin{proposition}\label{EndpointProjectionIsPeano}
If $H$ is a subgroup of $\pi_1(X,x_0)$, then the endpoint projection
$p_H\colon\widehat X_H \to X$ is a Peano map.
\end{proposition}
\begin{proof}
It suffices to show that for any $U$ open in $X$ the path component of any $
[\alpha]_H$ in $p_H^{-1}(U)$ is precisely $B_H([\alpha]_H,U)$. It's straightforward that $
B_H([\alpha ]_H,U)$ is path-connected so suppose $\beta $ is a path in $
p_H^{-1}(U)$ starting at $[\alpha]_H.$ We wish to show that
$\beta ([0,1])\subset B_H([\alpha]_H,U).$ Let $T=\{t:\beta (t)\in B_H([\alpha]_H,U)\}.$ Now $T$ is
nonempty since $\beta (0)=[\alpha]_H$ and open as the inverse image of an open set.
It suffices to prove $[0,t)\subset T$ implies $[0,t]\subset T$.
 Set $\beta (t)=[b]_H.$
Now $p_H\beta ([0,1])\subset U$ so in particular $p_H([b]_H)\in U.$ Consider $
B_H([b]_H,U).$ There is an $\varepsilon >0$ such that $\beta (t-\varepsilon
,t]\subset B_H([b]_H,U).$ Pick $s\in (t-\varepsilon
,t)$.  Then $\beta (s)=[c_{1}]_H$ and $
[b]_H=[b_{1}]_H$ such that $c_{1}\simeq b_{1}\ast \gamma _{1}$ for some $\gamma
_{1}$ with $\gamma _{1}[0,1]\subset U.$ But $\beta (s)\in B_H([\alpha]_H,U)$ so $
\beta (s)=[c_{2}]_H$ and $[\alpha]_H=[a_{1}]_H$ such that $c_{2}\simeq a_{1}\ast
\gamma _{2}$ for some $\gamma _{2}$ with $\gamma _{2}([0,1])\subset U.$ Then $
b\simeq_H b_{1}\simeq c_{1}\ast \gamma _{1}^{-1}\simeq_H c_{2}\ast \gamma
_{1}^{-1}\simeq a_{1}\ast \gamma _{2}\ast \gamma _{1}^{-1}\simeq_H a\ast
\gamma _{2}\ast \gamma _{1}^{-1}$ and $(\gamma _{2}\ast \gamma
_{1}^{-1})([0,1])\subset U$ so $[b]_H\in B_H([\alpha]_H,U)$ and $t\in T.$ Therefore $
T=[0,1].$
\end{proof}

In analogy to path lifting and unique path lifting properties (see \ref{PathLiftingDef}
and \ref{UniquenessOfPathLiftsDef})
one can introduce the corresponding concepts for hedgehogs:

\begin{definition}\label{HedgehogLiftingDef}
 A surjective map $f\colon X\to Y$ has the {\bf hedgehog lifting property}
if for any map $\alpha\colon \bigvee\limits_{s\in S} I_s\to Y$ from a hedgehog and any $y_0\in f^{-1}(\alpha(0))$ there is a
continuous lift $\beta\colon \bigvee\limits_{s\in S} I_s\to X$ of $\alpha$ such that $\beta(0)=y_0$.
\end{definition}

\begin{definition}\label{UniquenessOfHedgehogLiftsDef}
 $f\colon X\to Y$ has the {\bf unique hedgehog lifting property}
if it has both the hedgehog lifting property and the uniqueness of path lifts property.
\end{definition}

\begin{theorem}\label{HedgehogPeanoTheorem}
If $f\colon X\to Y$ has the unique hedgehog lifting property, then
\par
\noindent
$f\colon lpc(X)\to Y$ is a Peano map.
\end{theorem}
\proof Assume $U$ is open in $X$ and $x_0\in U$.
Suppose for each neighborhood $V$ of $f(x_0)$ in $X$ there is a path
$\alpha_V\colon (I,0)\to (f^{-1}(V),x_0)$ such that $\alpha_V(1)\notin U$.
By \ref{BasicHedgeHogLemma} the wedge $\bigvee\limits_{V\in S}f\circ\alpha_V$
is a map $g$ from a hedgehog to $Y$ (here $S$ is the family of all neighborhoods
of $f(x_0)$ in $Y$). Its lift must be the wedge $h=\bigvee\limits_{V\in S}\alpha_V$.
However $h^{-1}(U)$ is not open in $lpc(X)$, a contradiction.
\endproof

\begin{definition}\label{PeanoFunctorDef}
Given a map $f\colon X\to Y$ of topological spaces
its {\bf Peano map} $P(f)\colon P_f(X)\to Y$ is $f$ on $X$ equipped with
the topology generated by path components of sets $f^{-1}(U)$,
$U$ open in $Y$.
\end{definition}

Notice that in the case of $f=id_X$ the range $P_{id_X}(X)$ of $P(id_X)$, where $id_X\colon X\to X$
is the identity map, is identical to $lpc(X)$ as defined in
 \ref{LPCExistsThm}.

Recall $f\colon X\to Y$ is a {\bf Hurewicz fibration} if
every commutative diagram

$$
\begin{CD}
K\times \{0\}  @> \alpha  >>  X    \\
@V VV                               @VV f V \\
K\times I   @> H >>   Y
\end{CD}
$$
\noindent
has a filler $G\colon K\times I\to X$ (that means $f\circ G=H$
and $G$ extends $\alpha)$.
If the above condition is satisfied for $K$ being any $n$-cell $I^n$, $n\ge 0$
(equivalently, for any finite polyhedron $K$), then $f$
is called a {\bf Serre fibration}. Notice for $K$ being a point
this is the classical {\bf path lifting property}.

If the above condition is satisfied for $K$ being any hedgehog, then $f$
is called a {\bf hedgehog fibration}. 
If the above condition is satisfied for $K$ being any Peano space, then $f$
is called a {\bf Peano fibration}.

We will modify those concepts for maps between pointed spaces as follows:

\begin{definition}\label{SerreFibrationPropDef}
A map $f\colon (X,x_0)\to (Y,y_0)$ is a {\bf Serre $1$-fibration}
if any commutative diagram
$$
\begin{CD}
(I\times \{0\},(\frac{1}{2},0))  @> \alpha  >>  (X,x_0)    \\
@V VV                               @VV f V \\
(I\times I,(\frac{1}{2},0))   @> H >>   (Y,y_0)
\end{CD}
$$

\noindent has a filler $G\colon (I\times I,(\frac{1}{2},0))\to (X,x_0)$ (that means $f\circ G=H$
and $G$ extends $\alpha)$.
\end{definition}

Observe Serre $1$-fibrations have the path lifting property
in the sense that any path in $Y$ starting at $y_0$ lifts to a path in $X$
originating at $x_0$.

\begin{theorem}\label{MainCoveringTheorem}
Suppose $$
\begin{CD}
(T,z_0)  @> g_1  >>  (X,x_0)    \\
@V i VV                               @VV f V \\
(Z,z_0)   @> g >>   (Y,y_0)
\end{CD}
$$
is a commutative diagram in the topological category
such that $(Z,z_0)$ is a Peano space and $i$ is the inclusion from a path-connected
subspace $T$ of $Z$.
If $f$ is a Serre $1$-fibration, then there is a continuous lift $h\colon (Z,z_0)\to (P_f(X),x_0)$ of $g$ extending $g_1$ if the image of $\pi_1(g)\colon \pi_1(Z,z_0)\to \pi_1(Y,y_0)$
is contained in the image of $\pi_1(f)\colon \pi_1(X,x_0)\to \pi_1(Y,y_0)$.
\end{theorem}
\proof For each point $z\in Z$ pick a path $\alpha_z$ in $Z$ from $z_0$ to $z$
and let $\beta_z$ be a lift of $g\colon\alpha_z\mapsto Y$. In case of $z=z_0$ we pick the constant paths $\alpha_z$ and $\beta_z$. In case $z\in T$ 
the path $\alpha_z$ is contained in $T$
and $\beta_z=g_1\circ\alpha_z$.
Define $h\colon (Z,z_0)\to (P_f(X),x_0)$ by $h(z)=\beta_z(1)$.
Given a neighborhood $U$ of $g(z)$ in $Y$, let $V$ be the path component
of $h(z)$ in $f^{-1}(U)$ and let $W$ be the path component
of $g^{-1}(U)$ containing $z$. Our goal is to show $h(W)\subset V$
as that is sufficient for $h\colon (Z,z_0)\to (P_f(X),x_0)$ to be continuous.
For any $t\in W$ choose a path $\mu_t$ in $W$ from $z$ to $t$.
Let $\gamma$ be a loop in $X$ at $x_0$ so that $f(\gamma)$ is homotopic
to $g(\alpha_z\ast\mu_t\ast\alpha_t^{-1})$. Notice $f(\beta_z)$
is homotopic to $f(\gamma\ast \beta_t)$ via a homotopy $H$ so that
$H(\{1\}\times I)\subset U$. By lifting that homotopy to $X$
we get a path in $f^{-1}(U)$ from $h(z)$ to $h(t)$, i.e., $h(t)\in V$.
\endproof

\begin{corollary}\label{PeanoMapsAndFibrations}
A Peano map $f\colon X\to Y$ is a Peano fibration if and only if it is
a Serre $1$-fibration.
\end{corollary}
\proof Assume $f\colon X\to Y$ is a Peano map and a Serre $1$-fibration
(in the other direction \ref{PeanoMapsAndFibrations} is left as an exercise),
$g\colon Z\times\{0\}\to X$ is a map from a Peano space,
and $H\colon Z\times I\to Y$ is a homotopy starting from $f\circ g$.
Pick $z_0\in Z$ and put $x_0=g(z_0,0)$, $y_0=f(x_0)$.
Notice the image of $\pi_1(g)\colon \pi_1(Z\times \{0\}, (z_0,0))\to \pi_1(Y,y_0)$
is contained in the image of $\pi_1(f)$.
Use
 \ref{MainCoveringTheorem} to produce an extension $G\colon Z\times I\to X$
 of $g$ that is a lift of $H$.
\endproof

\section{Peano covering maps}\label{SECTION Peano-coverings}

\ref{MainCoveringTheorem} suggests the following concept:

\begin{definition}\label{PeanoCoveringDef}
A map $f\colon X\to Y$ is called a {\bf Peano covering map}
if the following conditions are satisfied:
\begin{enumerate}
\item $f$ is a Peano map,
\item $f$ is a Serre fibration,
\item The fibers of $f$ have trivial path components.
\end{enumerate}
\end{definition}

Notice 3) above can be replaced by $f$ having the unique path lifting property
(see \ref{PointedSerrePlusFibersImpliesUPLPLem}).
Also notice that, in case fibers of a Peano map $f\colon X\to Y$ are $T_0$ spaces,
path-components of fibers are trivial. Indeed, two points
in a path-component of a fiber are always in any open set that contains one of them.

\begin{proposition}\label{ProductOfPeanoCoveringMaps}
Any product of Peano covering maps is a Peano covering map.
\end{proposition}
\proof 
Suppose $f_s\colon X_s\to Y_S$, $s\in S$, are Peano covering maps.
Put $f=\prod\limits_{s\in S}f_s$, $X=\prod\limits_{s\in S}X_s$,
and $Y=\prod\limits_{s\in S}Y_s$.
By \ref{ProductOfPeanoMaps} $f$ is a Peano map. It is obvious $f$
is a Serre fibration and has the uniqueness of path lifting property.
\endproof

\begin{corollary}\label{MainPropertyOfPeanoCoverings}
Suppose $f\colon (X,x_0)\to (Y,y_0)$ is a Peano covering map.
If $(Z,z_0)$ is a Peano space, then any map $g\colon (Z,z_0)\to (Y,y_0)$ has a unique continuous lift $h\colon (Z,z_0)\to (X,x_0)$ if the image of $\pi_1(g)$
is contained in the image of $\pi_1(f)$. 
\end{corollary}
\proof By \ref{MainCoveringTheorem} a lift $h$ exists and is unique by the uniqueness of
path lifting property.
\endproof

Our basic example of Peano covering maps is related to the basic topology:

\begin{theorem}\label{ProjIsPeanoCMCharThm}
If $X$ is a path-connected space and $x_0\in X$, then the following conditions are equivalent:
\begin{itemize}
\item[a.] $p_H\colon (\widehat X_H,\widehat x_0)\to (X,x_0)$ has the unique path lifting property,
\item[b.]  $p_H\colon \widehat X_H\to X$ is a Peano covering map.
\end{itemize}
\end{theorem}
\proof a)$\implies$b). In view of \ref{EndpointProjectionIsPeano} 
and \ref{PointedSerreImpliesSerreLem} it suffices to show
$p_H\colon (\widehat X_H,\widehat x_0)\to (X,x_0)$ is a Serre fibration.
Suppose $f\colon (Z,z_0)\to (X,x_0)$ is a map from a simply connected Peano space $Z$
(the case of $Z=I^n$ is of interest here). There is a standard lift $g\colon (Z,z_0)\to
\widehat X_H$ of $f$ defined as $g(z)=[\alpha_z]_H$, where $\alpha_z$ is a path
in $Z$ from $z_0$ to $z$. If $T$ is a path-connected subspace of $Z$ containing $z_0$
and $h\colon (T,z_0)\to (\widehat X_H,\widehat x_0)$ is any continuous lift of $f|T$,
then $h=g|T$ due to the uniqueness of the path lifting property of $p_H$.
That proves $p_H$ is a Serre fibration in view of \ref{PointedSerreImpliesSerreLem}.
\par
b)$\implies$a) is obvious.
\endproof

\begin{theorem}\label{CharacterizationOfPeanoCoverings}
If $f\colon X\to Y$ is a map and $X$ is an lpc-space, then the following conditions are
equivalent:
\begin{itemize}
\item[a)] $f$ is a Peano covering map,
\item[b)] $f$ is a Peano fibration and has the uniqueness of path lifting property,
\item[c)] $f$ is a hedgehog fibration and has the uniqueness of path lifting property,
\item[d)] For any $x_0\in X$ and any map $g\colon (Z,z_0)\to (Y,f(x_0))$
from a simply-connected Peano space there is a lift
$h\colon (Z,z_0)\to (X,x_0)$ of $g$ and that lift is unique.
\end{itemize}
\end{theorem}
\proof a)$\implies$b). Suppose $H\colon Z\times I\to Y$ is a homotopy, $Z$
is a Peano space, and $G\colon Z\times\{0\}\to X$ is a lift of
$H|Z\times\{0\}$. Pick $z_0\in Z$, put $x_0=G(z_0,0)$ and $y_0=f(x_0)$,
and notice $im(\pi_1(Z\times I,(z_0,0)))\subset im(\pi_1(f))$.
Using \ref{MainCoveringTheorem} there is a lift of $H$ and that lift
is unique, hence it agrees with $G$ on $Z\times\{0\}$.
\par b)$\implies$c) is obvious.
\par d)$\implies$c) is obvious.
\par a)$\implies$b) follows from \ref{MainCoveringTheorem}.
\par c)$\implies$a). Notice $f$ has the unique hedgehog lifting property
and is a Serre $1$-fibration. By  \ref{HedgehogPeanoTheorem} $f$ 
is a Peano map.
\endproof

\begin{corollary}\label{CompositionOfPeanoCoverings}
Suppose $f\colon X\to Y$ and $g\colon Y\to Z$
are maps of path-connected spaces and $Y$ is a Peano space. If any two of $f$, $g$, $h=g\circ f$ are Peano
covering maps, then so is the third provided its domain is an lpc-space.
\end{corollary}
\proof In view of \ref{CharacterizationOfPeanoCoverings} it amounts
to verifying that the map has uniqueness of lifts of simply-connected Peano spaces,
an easy exercise.
\endproof

\begin{proposition}\label{LocalPeanoCoverings}
Suppose $f\colon X\to Y$ is a map.
\begin{itemize}
\item[a.] If $f\colon X\to Y$ is a Peano covering map, then $f\colon f^{-1}(U)\to U$
is a Peano covering map for every open subset $U$ of $Y$.
\item[b.]
If every point $y\in Y$ has a neighborhood $U$ such that
$f\colon f^{-1}(U)\to U$
is a Peano covering map, then $f$ is a Peano covering map.
\end{itemize}
\end{proposition}
\proof a). $f\colon f^{-1}(U)\to U$ is clearly a Peano map, is a fibration,
and has the unique path lifting property.
\par b). $f$ is a Serre $1$-fibration and path components of fibers are trivial.
If $V$ is an open subset of $Y$ containing $y$ we pick an open subset
$U$ of $X$ containing $f(y)$ such that $f\colon f^{-1}(U)\to U$ is a Peano covering map.
There is an open neighborhood $W$ of $f(y)$ in $U$ so that
the path component of $y$ in $f^{-1}(W)$ is open and is contained in $V\cap f^{-1}(U)$.
That proves $f\colon Y\to X$ is a Peano map.
\endproof

In analogy to regular classical covering maps let us introduce
regular Peano covering maps:

\begin{definition}\label{RegularPeanoCoverings}
A Peano covering map $f\colon X\to Y$ is {\bf regular} if
lifts of loops in $Y$ are either always loops of are always non-loops.
\end{definition}

\begin{corollary}\label{FZCoversArePeano}
Given a map $f\colon X\to Y$ the following conditions
are equivalent if $X$ is path-connected:
\begin{enumerate}
\item[a)] $f$ is a regular Peano covering map,
\item[b)] $f$ is a Peano covering map and the image of $\pi_1(f)$ is a normal subgroup of $\pi_1(Y,f(x_0))$
for all $x_0\in X$,
\item[c)] $f\colon X\to Y$ is a generalized covering map in the sense of Fischer-Zastrow.
\end{enumerate}
\end{corollary}
\proof a)$\implies$b). If the image of $\pi_1(f)$ is not a normal subgroup of 
$\pi_1(Y,f(x_0))$
for some $x_0\in X$, then there is a loop $\alpha$ in $Y$ at $y_0=f(x_0)$
that lifts to a loop in $X$ at $x_0$ and there is a loop $\beta$ in $Y$ at $y_0$
such that $\beta\ast\alpha\ast\beta^{-1}$ does not lift to a loop in $X$ at $x_0$.
Let $\gamma$ be a lift of $\alpha$ originating at $x_0$. Let $x_1=\beta(1)$.
Notice the lift of $\alpha$ originating at $x_1$ cannot be a loop,
a contradiction.

b)$\implies$c). As $im(\pi_1(f))$ is a normal subgroup $H$ of $\pi_1(Y,y_0)$,
it does nor depend on the choice of the base-point of $X$ in $f^{-1}(y_0)$.
Using \ref{MainCoveringTheorem} one gets $f$ is a generalized covering map.

c)$\implies$a). Since each hedgehog is contractible, $f$ has the unique hedgehog
lifting property and is a Peano map by \ref{HedgehogPeanoTheorem}.
It is also a Serre fibration, hence a Peano covering map. 
Also, as $im(\pi_1(f))$ is a normal subgroup $H$ of $\pi_1(Y,y_0)$,
it does nor depend on the choice of the base-point of $X$ in $f^{-1}(y_0)$.
Hence a loop in $Y$ lifts to a loop in $X$ if and only if it represents
an element of $H$. Thus $f$ is a regular Peano covering map.
\endproof

In the remainder of this section we will discuss the relation of Peano covering maps to
classical covering maps.

\begin{proposition}\label{PeanoCMIsTrivialBundle}
If $f\colon Y\to X$ is a Peano covering map and $U$ is an open
subset of $X$ such that every loop in $U$ is null-homotopic in $X$,
then $f^{-1}(V)\to P(V)$ is a a trivial discrete bundle for every path component
$V$ of $U$.
\end{proposition}
\proof Consider a path component $W$ of $f^{-1}(U)$ intersecting $f^{-1}(V)$.
$f$ maps $W$ bijectively onto $V$ and it is easy to see $f|W\colon W\to V$
is equivalent to $P(V)\to V$.
\endproof

\begin{corollary}\label{PeanoCovsForSemiSimple}
If $X$ is a semilocally simply connected Peano space, then
\par\noindent $f\colon Y\to X$
is a Peano covering map if and only if it is a classical covering map and $Y$ is connected.
\end{corollary}
\proof If $f$ is a classical covering map and $Y$ is connected, then $Y$ is locally
path-connected, $f$ has unique path lifting property and is a Serre $1$-fibration.
Thus it is a Peano covering map.
\par Suppose $f$ is a Peano covering map and $x\in X$. Choose
a path-connected neighborhood $U$ of $x$ in $X$ such that any loop
in $U$ is null-homotopic in $X$. By \ref{PeanoCMIsTrivialBundle}
$U$ is evenly covered by $f$.
\endproof

\begin{corollary}\label{PeanoAndClassicalCovers}
If $f\colon Y\to P(X)$ is a classical covering map, then
$f\colon Y\to X$ is a Peano covering map.
\end{corollary}
\proof By \ref{LocalPeanoCoverings}, $f\colon Y\to P(X)$ is a Peano covering map.
As the identity function induces a Peano covering map $P(X)\to X$,
$f\colon Y\to X$ is a Peano covering map by \ref{CompositionOfPeanoCoverings}.
\endproof

\begin{proposition}\label{CardinalityOfFibers}
If $f\colon Y\to X$ is a Peano covering map and $X$ is path-connected,
then all fibers of $f$ have the same cardinality.
\end{proposition}
\proof Given two points $x_1,x_2\in X$ fix a path $\alpha$ from $x_1$ to $x_2$
and notice lifts of $\alpha$ establish bijectivity of fibers $f^{-1}(x_1)$
and $f^{-1}(x_2)$.
\endproof

The following result has its origins in Lemma 2.3 of \cite{ConLam}
and Proposition 6.6 of \cite{FisZas}.

\begin{proposition}\label{RegularPeanoCMAreCCM}
Suppose $f\colon Y\to X$ is a regular Peano covering map. If $f^{-1}(x_0)$
is countable and $x_0$ has a countable basis of neighborhoods
in $X$, then there is a neighborhood $U$ of $x_0$ in $X$
such that $f^{-1}(V)\to P(V)$ is a classical covering map, where $V$ is the path component of $x_0$ in $U$.
\end{proposition}

\proof Switch to $X$ being Peano by considering $f\colon Y\to P(X)$. Notice $x_0$ has a countable basis of neighborhoods
and $f$ is open. Suppose there is no open subset $U$ of $X$ containing $x_0$
such that $U$ is evenly covered. That means path components of $f^{-1}(U)$
are not mapped bijectively onto their images.
\par
First, we plan to show there is a neighborhood $U$ of $x_0$ in $X$
such that the image of $\pi_1(U,x_0)\to \pi_1(X,x_0)$ is contained
in the image of $\pi_1(f)\colon \pi_1(Y,y_0)\to \pi_1(X,x_0)$.
In particular, there is a lift of $P(U,x_0)\to (Y,y_0)$ of the
inclusion induced map $P(U,x_0)\to (X,x_0)$.
\par Suppose no such $U$ exists.
By induction we will find a basis of neighborhoods $\{U_i\}$ of $x_0$ in $X$
and elements $[\alpha_i]\in \pi_1(U_i,x_0)$ that are not contained
in the image of $\pi_1(U_{i+1},x_0)\to \pi_1(X,x_0)$ and whose lifts
are not loops and end at points $y_i$ such that $y_i\ne y_j$ if $i\ne j$. Given a neighborhood $U_i$ pick a loop $\alpha_i$
in $(U_i,x_0)$ whose lift (as a path) in $(Y,y_0)$ is not a loop
and ends at $y_i\ne y_0$. There is a neighborhood $U_{i+1}$
of $x_0$ in $U_i$ such that the no path components of $f^{-1}(U_{i+1})$
contains both $y_0$ and some $y_{j}$, $j\leq i$. Pick a loop $\alpha_{i+1}$
in $(U_{i+1},x_0)$ whose lift is not a loop.
\par As in \cite{Paw} one can create infinite concatenations
$\alpha_{i(1)}\ast\ldots\ast \alpha_{i(k)}\ast\ldots$ for any increasing sequence
$\{i(k)\}_{k\ge 1}$. By looking at lifts of those infinite concatenations,
there are two different infinite concatenations
$\alpha_{i(1)}\ast\ldots\ast \alpha_{i(k)}\ast\ldots$
and
$\alpha_{j(1)}\ast\ldots\ast \alpha_{j(k)}\ast\ldots$ whose lifts end at the same point
$y\in f^{-1}(x_0)$. Pick the smallest $k_0$ so that $i(k_0)\ne j(k_0)$.
We may assume $i(k_0) < j(k_0)$ and conclude there are
loops $\beta$ in $(U_{k_0+1},x_0)$ and $\gamma$ in $(Y,y_0)$
so that $\alpha_{i(k_0)}\sim f(\gamma)\ast\beta$ i which case the lift
of $\alpha_{i(k_0)}$ in $(Y,y_0)$ ends in the path component of $f^{-1}(U_{i(k_0)+1})$ containing $y_0$, a contradiction.
\par As $f$ is a regular Peano covering map, we can find lifts
$(U,x_0)\to (Y,y)$ of the inclusion map $(U,x_0)\to (X,x_0)$
for any $y\in f^{-1}(x_0)$.
\endproof

\section{Peano subgroups}\label{SECTION Peano-subgroups}

\begin{definition}\label{PeanoSubgroupDef}
Suppose $(X,x_0)$ is a pointed path-connected space.
A subgroup $H$ of $\pi_1(X,x_0)$ is a {\bf Peano subgroup} of $\pi_1(X,x_0)$
if there is a Peano covering map $f\colon Y\to X$ such that $H$
is the image of $\pi_1(f)\colon \pi_1(Y,y_0)\to\pi_1(X,x_0)$
for some $y_0\in f^{-1}(x_0)$.
\end{definition}

\begin{proposition}\label{PeanoSubsAreHH}
If $H$ is a Peano subgroup of $\pi_1(X,x_0)$, then $X$ is homotopically
Hausdorff relative to $H$. In particular, $H$ is closed in $\pi_1(X,x_0)$
equipped with the basic topology.
\end{proposition}
\proof Choose a Peano covering map $f\colon Y\to X$ so that
$im(\pi_1(f))=H$ for some $y_0\in f^{-1}(x_0)$. If $g\in \pi_1(X,x_0)\setminus H$
and $\alpha$ is a path in $X$ from $x_0$ to $x_1$,
then lifts of $\alpha$ and $g\cdot \alpha$ end in two different points $y_1$
and $y_2$ of the fiber $f^{-1}(x_1)$ and there is a neighborhood $U$
of $x_1$ in $X$ such that no path component of
$f^{-1}(U)$ contains both $y_1$ and $y_2$.
Suppose there is a loop $\gamma$ in $(U,x_1)$ with the property
$[\alpha\ast\gamma\ast\alpha^{-1}]\in H\cdot g$. In that case
the lifts of both $\alpha\ast\gamma$ and $g\cdot\alpha$
end at $y_2$. Since the lift of $\alpha$ ends in the same path component
of $f^{-1}(U)$ as the lift of $\alpha\ast\gamma$, both $y_1$ and $y_2$
belong to the same component of $f^{-1}(U)$, a contradiction.
\par Use \ref{HGHausdorffAndBasicTopology} to conclude
$H$ is closed in $\pi_1(X,x_0)$
equipped with the basic topology.
\endproof

\begin{remark}\label{WeakPeanoSubsAreHHRemark}
In case of $H$ being the trivial subgroup, Lemma 2.10 of \cite{FisZas}
seems to imply that $X$ is homotopically
Hausdorff but the proof of it is valid only in a special case.
\end{remark}

\begin{proposition}\label{ConjugateOfPeanoSubs}
If $H$ is a Peano subgroup of $\pi_1(X,x_0)$, then
any conjugate of $H$ is a Peano
subgroup of $\pi_1(X,x_0)$.
\end{proposition}
\proof Choose a Peano covering map $f\colon Y\to X$ so that
$im(\pi_1(f))=H$ for some $y_0\in f^{-1}(x_0)$.
Suppose $G=g\cdot H\cdot g^{-1}$ and choose a loop $\alpha$
in $(X,x_0)$ representing $g^{-1}$. Let $\beta$ be a path in $(Y,y_0)$
that is the lift of $\alpha$. Put $y_1=\beta(1)$ and notice the image of
$\pi_1(f)\colon \pi_1(Y,y_1)\to \pi_1(X,x_0)$ is $G$.
\endproof

\begin{proposition}\label{IntroToBogley-SieradskiTopology}
Suppose $(X,x_0)$ is a pointed path-connected topological space.
If $f\colon (Y,y_0)\to (X,x_0)$ is a Peano covering map
with image of $\pi_1(f)$ equal $H$, then $f$ is equivalent to the projection
$p_H\colon \widehat X_H\to X$.
\end{proposition}
\proof
 Define $h\colon (\widehat X_H,\widehat x_0)\to (Y,y_0)$ by choosing a lift $\widehat \alpha$
of every path $\alpha$ in $X$ starting at $x_0$ and declaring $h([\alpha]_H)=\widehat\alpha(1)$. Note $h$ is a bijection.
Given $y_1=\widehat\alpha(1)$ and given a neighborhood $U$ of $y_1$ in $Y$
choose a neighborhood $V$ of $f(y_1)=\alpha(1)$ in $X$ so that the path component
of $f^{-1}(V)$ containing $y_1$ is a subset of $U$. Observe $B_H([\alpha]_H,V)\subset
h^{-1}(U)$ which proves $h$ is continuous.
\par Conversely, given a neighborhood $W$ of $\alpha(1)$ in $X$
the image $h(B_H([\alpha]_H,W))$ of $B_H([\alpha]_H,W)$ equals the path component
of $\widehat\alpha(1)$ in $f^{-1}(W)$ and is open in $Y$.
\endproof

\begin{theorem}\label{BasicPeanoCoveringThm}
If $X$ is a path-connected space, $x_0\in X$, and $H$ is a subgroup of $\pi_1(X,x_0)$, then the following conditions are equivalent:
\begin{itemize}
\item[a.] $H$ is a Peano subgroup of $\pi_1(X,x_0)$,
\item[b.] The endpoint projection $p_H\colon (\widehat X_H,\widehat x_0)\to X$ is a Peano covering map,
\item[c.] The image of $\pi_1(p_H)\colon \pi_1(\widehat X_H,\widehat x_0)\to \pi_1(X,x_0)$ is
contained in $H$,
\item[d.] $p_H\colon (\widehat X_H,\widehat x_0)\to (X,x_0)$ has the unique path lifting property.
\end{itemize}
\end{theorem}
\proof c)$\equiv$d) is done in \ref{ProjHasUPLP}.
b)$\equiv$d) is contained in \ref{ProjIsPeanoCMCharThm}.
\par
a)$\implies$b) follows from \ref{IntroToBogley-SieradskiTopology}.
\par
b)$\implies$a) holds as c) implies the image of $\pi_1(p_H)$ is $H$.
\endproof

Let us state a straightforward consequence of \ref{BasicPeanoCoveringThm}:

\begin{corollary}\label{SimplyPeanoCoveringThm}
If $X$ is a path-connected space and $x_0\in X$, then the following conditions are equivalent:
\begin{itemize}
\item[a.] The endpoint projection $p\colon \widehat X\to X$ is a Peano covering map,
\item[b.] $\pi_1(p)\colon \pi_1(\widehat X,\widehat x_0)\to \pi_1(X,x_0)$ is trivial,
\item[c.] $\widehat X$ is simply connected,
\item[d.] $p\colon (\widehat X,\widehat x_0)\to (X,x_0)$ has the unique path lifting property.
\end{itemize}
\end{corollary}

\begin{corollary}\label{ClosedNormalSubsArePeano}
Closed and normal subgroups of $\pi_1(X,x_0)$ are Peano subgroups of $\pi_1(X,x_0)$.
\end{corollary}
\proof By \ref{PiOneOfXHAndPathLifting} the endpoint projection $p_H\colon (\widetilde X_H,\widetilde x_0)\to X$ has unique path lifting property.
Since $p_H\colon (\widehat X_H,\widehat x_0)\to X$ has path lifting property,
this implies $p_H\colon (\widehat X_H,\widehat x_0)\to X$ has the unique path lifting
property.
\endproof

\begin{corollary}\label{IntersectionArePeano}
If $H(s)$ is a Peano subgroup of $\pi_1(X,x_0)$ for each $s\in S$,
then $G=\bigcap\limits_{s\in S}H(s)$ is a
 Peano subgroup of $\pi_1(X,x_0)$.
\end{corollary}
\proof The projection $p_G\colon (\widehat X_G,\widehat x_0)\to (X,x_0)$ factors through
\par\noindent
$p_{H(s)}\colon (\widehat X_{H(s)},\widehat x_0)\to (X,x_0)$ for each $s\in S$.
Therefore $im(\pi_1(p_G))\subset \bigcap\limits_{s\in S}H(s)=G$
and \ref{ProjIsPeanoCMCharThm} (in conjunction with \ref{ProjHasUPLP}) says $G$ is a Peano subgroup of $\pi_1(X,x_0)$.
\endproof

\begin{corollary}\label{UniversalPeanoCM}
For each path-connected space $X$ there is a universal Peano covering map
$p\colon Y\to X$. Thus, for each Peano covering map $q\colon Z\to X$
and any points $z_0\in Z$ and $y_0\in Y$ satisfying $q(z_0)=p(y_0)$,
there is a Peano covering map $r\colon Y\to Z$ so that $r(y_0)=z_0$.
Moreover, the image of $\pi_1(Y)$ is normal in $\pi_1(X)$.
\end{corollary}
\proof Let $H$ be the intersection of all Peano subgroups of $\pi_1(X,x_0)$
by \ref{IntersectionArePeano} and \ref{ConjugateOfPeanoSubs} it is a normal Peano subgroup of $\pi_1(X,x_0)$.
Put $Y=\widehat X_H$ and use \ref{MainPropertyOfPeanoCoverings}.
\endproof

It would be of interest to characterize path-connected spaces $X$
admitting a universal Peano covering that is simply connected (that amounts to
$\widehat X$ being simply connected). 
Here is an equivalent problem:

\begin{problem}\label{TrivialPeanoSubsProblem}
Characterize path-connected spaces $X$ so that the trivial group is a Peano subgroup
of $\pi_1(X,x_0)$.
\end{problem}

So far the following classes of spaces
belong to that category:
\begin{enumerate}
\item Any product of spaces admitting simply connected Peano cover (see \ref{ProductOfPeanoCoveringMaps}).
\item Subsets of closed surfaces: it is proved in \cite{FisZasFirst} that if $X$ is any
subset of a closed surface, then $\pi_1(X,x_0)\to\check\pi_1(X,x_0)$ is injective.
\item $1$-dimensional, compact and Hausdorff,
or $1$-dimensional, separable and metrizable: $\pi_1(X,x_0)\to\check\pi_1(X,x_0)$
is injective by \cite[Corollary 1.2 and Final Remark]{EdaKaw}. 
It is shown in \cite{Eda} (see proof of Theorem 1.4) that
the projection $\widehat X\to X$ has the uniqueness of path-lifting
property if $X$ is $1$-dimensional and metrizable.
See \cite{CanCon}
for results on the fundamental group of $1$-dimensional spaces.
\item Trees of manifolds: If $X$ is the limit of an inverse system of
closed PL-manifolds of some fixed dimension, whose consecutive terms are obtained
by connect summing with closed PL-manifolds, which in turn are trivialized by the
bonding maps, then X is called a tree of manifolds. Every tree of manifolds is
path-connected and locally path-connected, but it need not be semilocally simplyconnected
at any one of its points. Trees of manifolds arise as boundaries of certain
Coxeter groups and as boundaries of certain negatively curved geodesic spaces 
\cite{FisGui}.
It is shown in \cite{FisGui} that if X is a tree of manifolds (with a certain denseness of the
attachments in the case of surfaces), then $\pi_1(X,x_0)\to\check\pi_1(X,x_0)$ is injective.
\end{enumerate}

Notice Example 2.7 in \cite{FisZas} gives $X$ so that $p\colon \widehat X\to X$ does not
have the unique path lifting property (one can construct a simpler example
with $X$ being the Harmonic Archipelago). However, $X$ is not homotopically
Hausdorff.

\begin{problem}\label{HHProblem}
Is there a homotopically Hausdorff space $X$ such that $p\colon \widehat X\to X$
does not have the uniqueness of path lifting property?
\end{problem}

\begin{corollary}\label{CountableIndexNormalSubsOfPeanoArePeano}
Suppose $H$ is a normal subgroup of $\pi_1(X,x_0)$.
If there is a Peano subgroup $G$ of $\pi_1(X,x_0)$ containing $H$ such that $G/H$ is countable,
then $H$ is a
 Peano subgroup of $\pi_1(X,x_0)$ if and only if $X$ is homotopically Hausdorff
 relative to $H$.
\end{corollary}
\proof By \ref{PeanoSubsAreHH}, $X$ is homotopically Hausdorff
 relative to $H$ if $H$ is a
 Peano subgroup of $\pi_1(X,x_0)$.
 \par Suppose $X$ is homotopically Hausdorff
 relative to $H$. Given two lifts in $\widehat X_H$ of the same path in $X$,
 their composition with $\widehat X_H\to \widehat X_G$ are the same by
 \ref{BasicPeanoCoveringThm}. By
\ref{HAndPathLifting} the two lifts are identical and \ref{BasicPeanoCoveringThm}
says $H$ is a Peano subgroup of $\pi_1(X,x_0)$.
\endproof

\begin{corollary}\label{CountableNormalSubsArePeano}
Suppose $H$ is a normal subgroup of $\pi_1(X,x_0)$.
If $\pi_1(X,x_0)/H$ is countable,
then $H$ is a
 Peano subgroup of $\pi_1(X,x_0)$ if and only if $X$ is homotopically Hausdorff
 relative to $H$.
\end{corollary}

\section{Appendix: Pointed versus unpointed}\label{SECTION-PointedVsUnpointed}

In this section we discuss relations between pointed and unpointed lifting properties.

\begin{lemma}\label{PointedUPLPImpliesUPLPLem}
If $f\colon (X,x_0)\to (Y,y_0)$ has the uniqueness of path lifts property
and $X$ is path-connected, then $f\colon X\to Y$
has the uniqueness of path lifts property.
\end{lemma}
\proof Given two paths $\alpha$ and $\beta$ in $X$ originating at the same point
and satisfying $f\circ \alpha=f\circ\beta$, choose a path $\gamma$ in $X$
from $x_0$ to $\alpha(0)$. Now $f\circ (\gamma\ast\alpha)=f\circ (\gamma\ast\beta)$,
so $\gamma\ast\alpha=\gamma\ast\beta$ and $\alpha=\beta$.
\endproof

\begin{lemma}\label{PointedUPLPImpliesUPLPLem}
If $f\colon (X,x_0)\to (Y,y_0)$ has the unique path lifting property
and $X$ is path-connected, then $f\colon X\to Y$
has the unique path lifting property.
\end{lemma}
\proof In view of \ref{PointedUPLPImpliesUPLPLem} it suffices to show
$f\colon X\to Y$ is surjective and has the path lifting property.
If $y_1\in Y$, we pick a path $\alpha$ from $y_0$ to $y_1$ and lift it to
$(X,x_0)$. The endpoint of the lift maps to $y_1$, hence $f$ is surjective.
Suppose $\alpha$ is a path in $Y$ and $f(x_1)=\alpha(0)$.
Choose a path $\beta$ in $X$ from $x_0$ to $x_1$ and lift
$(f\circ \beta)\ast\alpha$ to a path $\gamma$ in $(X,x_0)$.
Due to the uniqueness of path lifts property of $f\colon (X,x_0)\to (Y,y_0)$
one has $\gamma(t)=\beta(2t)$ for $t\leq\frac{1}{2}$.
Hence $\gamma(\frac{1}{2})=x_1$ and $\lambda$ defined
as $\lambda(t)=\gamma(\frac{1}{2}+\frac{t}{2})$ for $t\in I$
is a lift of $\alpha$ originating from $x_1$.
\endproof

\begin{lemma}[Lemma 15.1 in \cite{Hu}]\label{PointedSerrePlusFibersImpliesUPLPLem}
If $f\colon X\to Y$
is a Serre $1$-fibration, then $f$ has the unique path lifting property if
and only if path components of fibers of $f$ are trivial.
\end{lemma}
\proof Suppose the fibers of $f$ have trivial path components
and $\alpha,\beta$ are two lifts of the same path in $Y$ that
originate at $x_1\in X$. Let $H\colon I\times I\to Y$ be the
standard homotopy from $f\circ (\alpha^{-1}\ast\beta)$ to the constant
path at $f(x_1)$. There is a lift $G\colon I\times I\to X$ of $H$
starting from $\alpha^{-1}\ast\beta$.
As path components of $f$ are trivial, $\alpha=\beta$ due to the way
the standard homotopy $H$ is defined.
\endproof

\begin{lemma}\label{PointedSerreImpliesSerreLem}
Suppose $n\ge 1$.
If $f\colon (X,x_0)\to (Y,y_0)$ is a Serre $n$-fibration,
both $X$ and $Y$ are path-connected, and $f$ has the uniqueness of path lifts property, then $f\colon X\to Y$
is a Serre $n$-fibration.
\end{lemma}
\proof Suppose $H\colon I^n\times I\to Y$ is a homotopy
and $G\colon I^n\times \{0\}\to X$ is its partial lift.
Choose a path $\alpha$ in $X$ from $x_0$ to $G(b,0)$, where $b$ is the center of $I^n$.
We can extend $G$ to a homotopy $G\colon I^n\times [-1,0]\to X$ starting from the constant map to $x_0$. By splicing $f\circ G$ with original $H$, we can extend $H$
to $H\colon I^n\times [-1,1]\to Y$. That $H$ can be lifted to $X$ and the lift
must agree with $G$ on $I^n\times [-1,0]$ due to the uniqueness of path lifts property of
$f$.
\endproof


\begin{thebibliography}{99}

\bibitem{BP3}
V. Berestovskii, C. Plaut, {\em Uniform universal covers of uniform spaces},
Topology Appl. 154 (2007), 1748--1777.


\bibitem{BogSie}  W.A.Bogley, A.J.Sieradski, {\em Universal path spaces}, http://oregonstate.edu/\symbol{126}bogleyw/\#research


\bibitem{BDLM1} N.Brodskiy, J.Dydak, B.Labuz, A.Mitra,
{\em Rips complexes and covers in the uniform category}, in preparation

\bibitem{BDLM2} N.Brodskiy, J.Dydak, B.Labuz, A.Mitra,
{\em Topological and uniform structures on the fundamental group}, in preparation

\bibitem{CanCon Big}
J.W. Cannon, G.R. Conner,
{\em The big fundamental group, big Hawaiian earrings, and the big free groups},
Topology and its Applications 106 (2000), no. 3, 273--291.

\bibitem{CanCon}
J.W. Cannon, G.R. Conner,
{\em On the fundamental groups of one-dimensional spaces},
Topology and its Applications 153 (2006), 2648--2672.

\bibitem{ConFea} G. R. Conner and D. Fearnley, {\em Fundamental groups of spaces which are not locally path connected},
preprint (1998).

\bibitem{ConLam} G. R. Conner and J. Lamoreaux, {\em On the existence of universal covering spaces for metric spaces and subsets of the Euclidean plane},
Fundamenta Math. 187 (2005), 95--110.


  \bibitem{DydSeg}
J. Dydak and J. Segal, {\em Shape theory: An introduction},
Lecture Notes in Math. 688, 1--150, Springer Verlag 1978.

\bibitem{Eda} K. Eda, {\em The fundamental groups of one-dimensional spaces
and spatial homomorphisms},
Topology and Its Applications, 123 (2002) 479--505.

\bibitem{EdaKaw} K. Eda and K. Kawamura, {\em The fundamental group of one-dimensional spaces},
Topology and Its Applications, 87 (1998) 163--172.

\bibitem{Fab} P.Fabel, {\em Metric spaces with discrete topological fundamental group}, Topology and its Applications 154 (2007), 635--638.

\bibitem{FisGui} H. Fischer and C.R. Guilbault, {\em On the fundamental groups of trees of manifolds},
Pacific Journal of Mathematics 221 (2005) 49--79.

\bibitem{FisZasFirst} H. Fischer, A. Zastrow, {\em The fundamental groups of subsets of closed surfaces
inject into their first shape groups}, Algebraic and Geometric Topology 5 (2005)
1655--1676.


\bibitem{FisZas}  H.Fischer, A.Zastrow, {\em Generalized universal
coverings and the shape group}, Fundamenta Mathematicae 197 (2007), 167--196.

\bibitem{HilWyl}
P.J. Hilton, S. Wylie, {\em Homology theory: An introduction to algebraic topology}, Cambridge University Press, New York 1960 xv+484 pp.

\bibitem{Hu} Sze-Tsen Hu, {\em Homotopy theory}, Academic Press,
New York and London, 1959.

\bibitem{Lim}
E. L. Lima, {\em Fundamental groups and covering spaces},
AK Peters, Natick, Massachusetts, 2003.

\bibitem{MarSeg}
S. Marde\v si\' c and J. Segal,
 {\em Shape theory}, North-Holland Publ.Co., Amsterdam 1982.

\bibitem{Mun}
J. R. Munkres, {\em Topology}, Prentice Hall, Upper Saddle River, NJ 2000.

\bibitem{Paw} J.Pawlikowski, {\em The fundamental group
of a compact metric space},
Proceedings of the
American Mathematical Society, 126 (1998), 3083--3087.

\bibitem{RepZas}
 D.Repov\v s and A.Zastrow, Shape injectivity is not implied by being
 strongly homotopically Hausdorff, preprint.

\bibitem{Spa}
E. Spanier,
{\em Algebraic topology}, McGraw-Hill, New York 1966.



\end{thebibliography}
\end{document}